\documentclass[11pt,final]{article}
\usepackage{a4}
\usepackage{amsmath}%
\usepackage{amstext}%
\usepackage{amssymb}%
\usepackage{amsthm}
\usepackage{comment}
\usepackage{listings}

\usepackage[final]{graphicx}
\usepackage{subcaption}

\usepackage[margin=0.5in]{geometry}

\usepackage{graphicx}
\usepackage{epstopdf}

\usepackage{xcolor}

\newtheorem{theorem}{Theorem}

\newtheorem{lemma}[theorem]{Lemma}

\newtheorem{rem}[theorem]{Remark}
\newenvironment{remark}{\rem\rm}{\endrem}

\newcommand{\R}{\mathbb{R}}%
\newcommand{\N}{\mathbb{N}}%
\newcommand{\e}{\varepsilon}%
\newcommand{\ol}{\overline}%
\newcommand{\ox}{\overline{x}}

\newcommand{\n}{{\nabla}}
\newcommand{\p}{{\partial}}

\newcommand{\To}{\longrightarrow}
\def\a{\alpha}

\def\e{\epsilon}

\def\l{\lambda}
\def\<{\langle}
\def\>{\rangle}
\DeclareMathOperator*\prox{prox}%
\DeclareMathOperator*\argmin{argmin}

\DeclareMathOperator*\pr{pr}

\title{A proximal-gradient inertial algorithm with Tikhonov regularization: strong convergence to the minimal norm solution}

\author{ Szil\'{a}rd Csaba L\'{a}szl\'{o} \thanks{Technical University of Cluj-Napoca, Department of Mathematics, Memorandumului 28, Cluj-Napoca,
 Romania, e-mail: szilard.laszlo@math.utcluj.ro. This work was supported by a grant of the Ministry of Research, Innovation and Digitization, CNCS-UEFISCDI, project number PN-III-P1-1.1-TE-2021-0138, within PNCDI III.}}

\begin{document}
\maketitle

\noindent \textbf{Abstract.} We investigate the strong convergence properties of a proximal-gradient inertial algorithm with two Tikhonov regularization terms in connection to the minimization problem of the sum of a convex lower semi-continuous function $f$ and a smooth convex function $g$.  For the appropriate setting of the parameters we provide strong convergence of the generated sequence $(x_k)$ to the minimum norm minimizer of our objective function $f+g$. Further, we obtain fast convergence to zero of the objective function values in a generated sequence but also for the discrete velocity and the sub-gradient of the objective function. We also show that for another settings of the parameters the optimal rate of order $\mathcal{O}(k^{-2})$ for the potential energy $(f+g)(x_k)-\min(f+g)$ can be obtained.

\vspace{1ex}



\section{Introduction}\label{sec-intr}

Let $\mathcal{H}$ be a real Hilbert space endowed with the scalar product $\< \cdot,\cdot\>$ and norm $\|\cdot\|=\sqrt{\< \cdot,\cdot\>}$. Consider the optimization problem
\begin{equation}\label{opt-pb}\tag{P} \inf_{x\in\mathcal{H}}f(x)+g(x)
\end{equation}
where $f:\mathcal{H}\To \overline{\R}=\R\cup\{+\infty\}$ is a convex lower semi-continuous and $g:\mathcal{H}\To \R$ is a convex, continuously Fr\'{e}chet differentiable function, with $L$-Lipschitz continuous gradient. We assume that the set of minimizers of $f+g$, that is $\argmin(f+g)$, is nonempty.
In order to introduce a proximal-gradient algorithm associated to the optimization problem \eqref{opt-pb} consider a sequence $(t_k)_{k\ge0}$ with $t_0=1$ and that for $k\ge 1$  satisfies the condition
\begin{equation}\label{condT}\tag{T}
t_{k-1}<t_{k}<\frac{1+\sqrt{1+4t_{k-1}^2}}{2}.
\end{equation}

Let $(t_k)_{k\ge0}$ a sequence that satisfies \eqref{condT} and consider the following  inertial proximal-gradient algorithm.
Let $x_0,x_{1}\in\mathcal{H}$
and for all $k\ge 1$ set

\begin{equation}\label{tdiscgen}\tag{TIREPROG}
\left\{\begin{array}{lll}
y_k= x_k+\frac{(t_{k}-1)(t_{k-1}-1)}{t_{k-1}^2}(x_k-x_{k-1})-\frac{-t_{k}^2+t_{k}+t_{k-1}^2}{t_{k-1}^2t_{k}} x_k
\\
x_{k+1}=\prox_{sf}(y_k-s\n g(y_k)-s\e_k y_k).
\end{array}\right.
\end{equation}
Here $\prox\nolimits_{s  f} : {\mathcal{H}} \to \mathcal{H}, \quad \prox\nolimits_{s f}(x)=\argmin_{y\in \mathcal{H}}\left \{f(y)+\frac{1}{2s}\|y-x\|^2\right\},$ denotes the proximal point operator of the convex function $s  f$.
We assume that  $s\in\left(0,\frac{1}{L}\right)$ and that $(\e_k)_{k\ge 1}$ is a non-increasing positive sequence that satisfies $\lim_{k\to+\infty}\e_k=0,$ hence the term $s\e_ky_k$ in \eqref{tdiscgen} is a Tikhonov regularization term. A comprehensive analysis of condition \eqref{condT} will carried out in section \ref{onT}. We mention however, that according to the definition of $t_k$, the inertial parameter $\frac{(t_{k}-1)(t_{k-1}-1)}{t_{k-1}^2}$ is non-negative and the parameter $\frac{-t_{k}^2+t_{k}+t_{k-1}^2}{t_{k-1}^2t_{k}}$, (which will play the role of another Tikhonov regularization parameter in algorithm \eqref{tdiscgen}), is positive for all $k\ge 1.$

\subsection{Motivation and related works}\label{motivation}

The condition \eqref{condT} imposed on  the sequence $(t_k)$ is inspired by  Nesterov's convex gradient method \cite{Nest1} associated to the optimization problem $\inf_{x\in\mathcal{H}}g(x)$, where the objective function $g$ is convex, smooth and has an $L-$Lipschitz continuous gradient, that is: $x_0=x_1\in \mathcal{H}$ and for all $k\ge 1$
\begin{equation}\label{orignest}
\left\{\begin{array}{lll}
y_k= x_k+\frac{t_k-1}{t_{k+1}}(x_k-x_{k-1})
\\
x_{k+1}=y_k-s\n g(y_k).
\end{array}\right.
\end{equation}

According to \cite{Nest1}, the sequences generated by \eqref{orignest} satisfy
$g(x_k)-\min g=\mathcal{O}(k^{-2})\mbox{ as }k\to+\infty$, provided the stepsize $s\in\left(0,\frac{1}{L}\right]$ and the sequence $(t_k)$ is defined via the recursion $t_1=1,\,t_{k+1}=\frac{1+\sqrt{1+4t_k^2}}{2},$ for all $k\ge 1.$ We underline that the rate $g(x_k)-\min g=\mathcal{O}(k^{-2})\mbox{ as }k\to+\infty$ obtained by Nesterov is optimal in the class of convex smooth functions with Lipschitz continuous gradient. However the convergence of the sequences generated by \eqref{orignest}, (at least in the weak topology of $\mathcal{H}$), is still an open problem.

The results of Nesterov were extended by Beck and Teboulle  to the optimization problem \eqref{opt-pb}, (see \cite{BT}), where  the following proximal-gradient inertial algorithm, named (FISTA), was considered:  $x_0=x_1\in \mathcal{H}$ and for all $k\ge 1$
\begin{equation}\label{fista}
\left\{\begin{array}{lll}
y_k= x_k+\frac{t_k-1}{t_{k+1}}(x_k-x_{k-1})
\\
x_{k+1}=\prox_{sf}(y_k-s\n g(y_k)).
\end{array}\right.
\end{equation}
Also in \eqref{fista} one has $s\in\left(0,\frac{1}{L}\right]$ and the sequence $(t_k)$ satisfies the recursion $t_1=1,\,t_{k+1}=\frac{1+\sqrt{1+4t_k^2}}{2},$ for all $k\ge 1.$ According to \cite{BT}, for the sequence generated by \eqref{fista}  the rate
$(f+g)(x_k)-\min(f+ g)=\mathcal{O}(k^{-2})\mbox{ as }k\to+\infty$ holds.
Also in this case, the convergence of the sequences generated by  \eqref{fista} is still a widely open problem.

Fortunately, in order to obtain the optimal rate of order $\mathcal{O}(k^{-2})$ in \eqref{orignest} or \eqref{fista} it is enough to assume that the sequence $(t_k)$ satisfies the condition \eqref{condT}. More precisely, if one consider $t_k=r(k-1),\,r\le\frac12,$ then $(t_k)$ satisfies condition \eqref{condT}, i.e., one has $t_k<t_{k+1}<\frac{1+\sqrt{1+4t_k^2}}{2},$ therefore in the literature $t_k$ is usually taken in the form $t_k=\frac{k-1}{\a-1},\,\a\ge 3.$ In that case algorithm \eqref{fista} becomes: $x_0=x_1\in\mathcal{H}$ and for $k\ge1$
\begin{equation}\label{modfista}
x_{k+1}=\prox\nolimits_{sf}\left(x_k+\left(1-\frac{\a}{k}\right)(x_k-x_{k-1})-s\n g\left(x_k+\left(1-\frac{\a}{k}\right)(x_k-x_{k-1})\right)\right).
\end{equation}
According to \cite{AP,CD} the sequences generated by \eqref{modfista} converge in the weak topology of $\mathcal{H}$ to a minimizer of $f+g$ and $(f+g)(x_k)-\min (f+g)=o(k^{-2})\mbox{ as }k\to+\infty$, provided $\a>3.$ Note that for $\a=3$ one has only the rate $(f+g)(x_k)-\min (f+g)=\mathcal{O}(k^{-2})\mbox{ as }k\to+\infty$ and, also in this case it is not known whether the generated sequences converge.

In \cite{ACCR} the authors considered  a FISTA type algorithm with a Tikhonov regularization term, that is,
\begin{equation}\label{fistaTikh}
\left\{\begin{array}{lll}
y_k= x_k+\frac{t_k-1}{t_{k+1}}(x_k-x_{k-1})
\\
x_{k+1}=\prox_{sf}(y_k-s\n g(y_k)-s\e_ky_k).
\end{array}\right.
\end{equation}
Here $(\e_k)$ is a non-increasing positive sequence that goes to zero.
In case $(t_k)$ satisfies condition \eqref{condT}, (even with equality on the right hand side),  then  according to \cite{ACCR} the  ergodic strong convergence result 
$$\lim_{k\to+\infty}\left\|\frac{\sum_{i=1}^k  \frac{\e_i}{t_{i+1}}x_i}{\sum_{i=1}^k\frac{\e_i}{t_{i+1}}}-x^*\right\|=0$$
holds, where $x^*$ is the minimum norm minimizer of $f+g$.
Fast convergence of the discrete velocity $\|x_k-x_{k-1}\|$ is also obtained, but no convergence result for the potential energy $(f+g)(x_k)-\min (f+g)$ is provided. According to \cite{ACCR}, obtaining both the fast convergence of the potential energy $(f+g)(x_k)-\min (f+g)$
 and the convergence of the generated sequences toward the  minimum norm solution is a difficult challenge, because these two requirements are somewhat incompatible. 
 
 In this paper we answer positively to this challenge, by obtaining even something more: both "full" strong convergence of the generated sequences to the minimum norm solution, that is, $\lim_{k\to+\infty}\|x_k-x^*\|=0,$ and fast rates for the decay $(f+g)(x_k)-\min (f+g)$.
Therefore, in this paper we associated to the optimization problem \eqref{opt-pb} the algorithm \eqref{tdiscgen}, which is a   proximal-gradient algorithm with two Tikhonov regularization terms, and our goal is to provide conditions such that the sequences generated by this algorithm converge, in the strong topology of $\mathcal{H}$, to the minimum norm minimizer of the objective function $f+g.$  At the same time we aim to preserve (as much as possible) the rate $\mathcal{O}(k^{-2})$ for the  potential energy $(f+g)(x_k)-\min (f+g).$ That is the reason why, inspired from \cite{L-CNSN} and \cite{KL-amop}, we considered in our inertial proximal-gradient algorithm two Tikhonov regularization terms.  Though considering two Tikhonov regularization terms  increase the complexity of our algorithm, the use of both Tikhonov regularization terms is essential in order to obtain strong convergence of the generated sequences to the minimum norm minimizer of the objective function, as some numerical experiments show. Indeed, as we mentioned before, the term $s\e_ky_k$ is a Tikhonov regularization term since $\e_k$ is a nonincreasing positive sequence that goes to $0$ as $k\to+\infty.$ At the same  time, under some extra assumptions on  $(t_k)$ that will assure  strong convergence of the generated sequences,  the term $\frac{-t_{k}^2+t_{k}+t_{k-1}^2}{t_{k-1}^2t_{k}} x_k$ can be considered as a  Tikhonov regularization term since $\frac{-t_{k}^2+t_{k}+t_{k-1}^2}{t_{k-1}^2t_{k}}$ is a positive sequence that goes to $0$ as $k\to+\infty$, see Remark \ref{afterstrconv}.

 We emphasize that the introduction of the Tikhonov regularization terms in the classical proximal-gradient algorithm will assure the strong convergence of the generated sequences to the element of minimal norm from $\argmin( f+g),$ (for similar results see \cite{AL-siopt,abc2,ACCR,ACR,att-com1996,AC,AL-nemkoz,BCL,BGMS,CPS,JM-Tikh,L-jde,L-CNSN,Tikh,TA}). Our analysis reveals that the inertial parameter and the Tikhonov regularization parameters, (actually $t_k$ and $\e_k$), are strongly correlated. This fact is in concordance with some recent results from the literature concerning the strong convergence of the trajectories of some continuous second order dynamical systems to a minimal norm minimizer of a convex function or to the minimal norm zero of a maximally monotone operator \cite{AL-siopt,ABCR, ABCR_0,ACR,ACR2,AL-nemkoz,BCL,BCLstr,BGMS,K,L-jde,L-coap}.  Due to this correlation, the inertial parameter in \eqref{tdiscgen} is  $\frac{(t_{k}-1)(t_{k-1}-1)}{t_{k-1}^2}$ which is quite different from the inertial parameter used in \eqref{fista}, that is $\frac{t_k-1}{t_{k+1}},$ however in order to obtain the desired rates, both inertial parameters must go to $1$ as $k\to+\infty.$

We underline that the forms of the inertial parameter $\frac{(t_{k}-1)(t_{k-1}-1)}{t_{k-1}^2}$ and the Tikhonov regularization parameter $\frac{-t_{k}^2+t_{k}+t_{k-1}^2}{t_{k-1}^2t_{k}}$ are crucial in order to obtain our strong convergence result.

 In \cite{AL-nemkoz} the following inertial-proximal algorithm was considered in connection to the smooth optimization problem $\inf_{x\in\mathcal{H}}f(x)$: $x_0,x_{1}\in\mathcal{H}$ and for all $k\ge 1$ set
\begin{equation}\label{algoAL}
x_{k+1}={ \rm prox}_{f}\left( x_k+\left(1-\frac{\a}{k}\right)(x_k -  x_{k-1}) - \frac{c}{k^2}x_k\right).
\end{equation}
where  $\a>3$ and $c>0$. To our best knowledge this is the first inertial algorithm  in the literature for which both the strong convergence result $\liminf_{k\to+\infty}\|x_k-x^*\|=0$ for the generated sequences and fast convergence of the potential energy $f(x_k)-\min f$ and discrete velocity $\|x_k-x_{k-1}\|$ were obtained. Note that \eqref{algoAL} can be obtained via implicit discretization from the following dynamical system studied in \cite{AL-nemkoz,ACR}.
\begin{align}\label{DynSysACR}
&\ddot{x}(t) + \frac{\a}{t} \dot{x}(t) +\nabla f\left(x(t)\right) +\e(t)x(t)=0,\,
x(t_0) = u_0, \,
\dot{x}(t_0) = v_0,
\end{align}
where $t\ge t_0 > 0,\a\ge 3$, $(u_0,v_0) \in \mathcal{H} \times \mathcal{H}$ and the Tikhonov regularization parameter $\e(t)$ is a nonincreasing positive function satisfying $\lim_{t\to+\infty}\e(t)=0$. However the rates are not entirely preserved, since in \eqref{algoAL} one obtains only the rates $f(x_k)-\min f=o(k^{-2s})$ and  $\|x_k-x_{k-1}\|=o(k^{-s})$ as $k\to+\infty$, where $s\in\left[\frac12,1\right),$ meanwhile according to \cite{AL-nemkoz}, related to the dynamical system \eqref{DynSysACR} the rates $f(x(t))-\min f=\mathcal{O}(t^{-2})$ and  $\|\dot{x}(t)\|=o(t^{-1})$ are obtained, where $x(t)$ is a trajectory generated by the dynamical system \eqref{DynSysACR}.

In order to show the strong convergence result $\lim_{k\to+\infty}\|x_k-x^*\|=0$ but also the correlation among the stepsize, inertial parameter and Tikhonov regularization parameter, in \cite{L-CNSN} the author assumed that the objective function $f$ in $\inf_{x\in\mathcal{H}}f(x)$ is proper, convex and lower semicontinuous only and associated to this optimization problem the following inertial-proximal algorithm: $x_0,x_{1}\in\mathcal{H}$ and for all $k\ge 1$ set
\begin{equation}\label{algostrongL}
x_{k+1}={ \rm prox}_{\l_k f}\left( x_k+\left(1-\frac{\a}{k^q}\right)(x_k -  x_{k-1}) - \frac{c}{k^p}x_k\right),
\end{equation}
where  $\a,\,q,\,c,\,p>0$ and $(\l_k)$ is a sequence of positive real numbers.  According to \cite{L-CNSN}, in case the stepsize $\l_k\equiv 1$ and $0<q<1,\,1<p<q+1$ then $\lim_{k\to+\infty}\|x_k-x^*\|=0$, where $x^*$ is the minimal norm element from $\argmin f$. Further, $\|x_k-x_{k-1}\|\in \mathcal{O}(k^{-\frac{q+1}{2}})\mbox{ as }k\to+\infty$ and $f(x_{k})-\min f= \mathcal{O}(k^{-p})\mbox{ as }k\to+\infty.$
If   $q+1<p\le 2$ and for $p=2$ one has $c>q(1-q)$,
then  $(x_k)$ converges weakly to a minimizer of $f.$ Further,
$f(x_k)-\min f=\mathcal{O}(k^{-q-1}) \mbox{ and }\|x_{k}-x_{k-1}\|=\mathcal{O}(k^{-\frac{q+1}{2}})\mbox{ as } k\to+\infty.$

Concerning the case of inertial gradient type algorithms associated to a an optimization problem with smooth objective function with L-Lipschitz continuous gradient, in \cite{KL-amop} the authors considered a Nesterov type algorithm with two Tikhonov regularization terms, that is,
for the starting points $x_0,x_{1}\in\mathcal{H}$ and for all $k\ge 1$ set
\begin{equation}\label{tdiscgenKL}
\left\{\begin{array}{lll}
y_k= x_k+b_{k-1}(x_k-x_{k-1})-c_k x_k
\\
x_{k+1}=y_k-s\n f(y_k)-s\e_k y_k.
\end{array}\right.
\end{equation}
We emphasize that in \eqref{tdiscgenKL} the stepsize satisfies $0<s<\frac{1}{L}$ and $\e_k$ is a nonincreasing positive sequence that goes to $0$ just as in \eqref{tdiscgen}, however, the inertial parameter $b_{k-1}$ and the Tikhonov regularization parameter $c_k$  depend by $\e_k$ and have  complex forms. In \cite{KL-amop}, strong convergence of the generated sequences $(x_k)$ and $(y_k)$ to the minimum norm minimizer of $f$ has been obtained, further fast convergence rates for the potential energies $f(x_k)-\min f$, $f(y_k)-\min f$, of order $\mathcal{O}(\e_k)$,  fast rates for the discrete velocity and value of the gradient in a generated sequence of order $o(\sqrt{\e_k})$ were provided.

The scope of this paper is to obtain similar results for the sequences generated by \eqref{tdiscgen}, that is, beside obtaining strong convergence to the minimum norm minimizer of our objective function $f+g$, we ought to provide fast convergence  rates for the potential energy $(f+g)(x_k)-\min (f+g)$ and discrete velocity $\|x_k-x_{k-1}\|.$

\subsection{On condition \eqref{condT} and a model result}\label{onT}

Let us discuss about condition \eqref{condT}.
Since $t_0=1$ and $t_{k-1}<t_k$ for all $k\ge 1$ we conclude that  the sequence $(t_k)_{k\ge 0}$ is increasing and $t_k>1$ for all $k\ge 1.$ Hence, the sequence $(t_k)_{k\ge 0}$ has a limit greater than $1.$ Note that if we would allow equality in the right hand side of \eqref{condT}, that is, $t_k$ is defined by the recursion $t_k=\frac{1+\sqrt{1+4t_{k-1}^2}}{2}$ then the limit of $t_k$ is $+\infty$ and that is the case of Nesterov's algorithm or FISTA. However, in order to define \eqref{tdiscgen} we need the parameter $\frac{-t_k^2+t_k+t_{k-1}^2}{t_{k-1}^2t_k} \neq 0$, hence we cannot allow the equality in \eqref{condT}.

 Nevertheless, if we assume that beside \eqref{condT} one has $\lim_{k\to+\infty}t_k=+\infty$ then it is obvious that $\lim_{k\to+\infty}\frac{t_k}{t_{k-1}}=1,$ since  \eqref{condT} provides
$$1<\frac{t_k}{t_{k-1}}<\frac{1}{2t_{k-1}}+\sqrt{1+\frac{1}{4t_{k-1}^2}}.$$

Consequently, the inertial parameter in algorithm \eqref{tdiscgen}, i.e. $\frac{(t_{k-1}-1)(t_k-1)}{t_{k-1}^2}$, goes to $1$ as $k\to+\infty.$ Further, the parameter $\frac{-t_k^2+t_k+t_{k-1}^2}{t_{k-1}^2t_k}$ goes to $0$ as $k\to+\infty,$ hence indeed in this case $\frac{-t_k^2+t_k+t_{k-1}^2}{t_{k-1}^2t_k}x_k$ is a Tikhonov regularization term.

From a practical point of view, the most important  sequences that satisfy \eqref{condT} are of the form $t_k=(ak+1)^q,\,k\ge 0$ where we assume that $\frac{1}{2}\le q\le 1$ and $0<a\le\frac{1}{2q}.$ Indeed, in this case $t_0=1$ and $t_k>1$ for all $k\ge 1.$ Further, since the function $\phi(x)=(ax+1)^q$ is increasing on the interval $[0,+\infty)$, it is obvious that $t_k>t_{k-1}$ for all $k\ge 1.$

Let us show that $t_k<\frac{1+\sqrt{1+4t_{k-1}^2}}{2}$ for all $k\ge 1,$ that is
$$-(ak+1)^{2q}+(ak+1)^q+(ak+1-a)^{2q}>0\mbox{ for all }k\ge 1.$$

If $q=\frac12$ then the claim follows directly.
Otherwise, consider the function $\phi(x)=(ax+1)^{2q}$, which, since $q< \frac12$, is strictly convex on the interval  $[0,+\infty)$. For $x\ge 1$, by using the gradient inequality, one has
$\phi(x-1)-\phi(x)> -\phi'(x)$ or equivalently, $(ax+1-a)^{2q}-(ax-a)^{2q}>-2qa(ax+1)^{2q-1}.$
Further, since $a\le \frac{1}{2q}$ and $q\le 1$ one has $(ax+1)^q\ge 2qa(ax+1)^{2q-1}$, hence $(ax+1-a)^{2q}-(ax-a)^{2q}+(ax+1)^q>0$ and the claim follows.

So take $t_k=(ak+1)^q,\,k\ge 0$ with $\frac{1}{2}\le q\le 1$ and $0<a\le\frac{1}{2q}$ and let $\e_k=\frac{c}{k^p},\,c,p>0.$ Then in this particular case \eqref{tdiscgen} has the form:
$x_0,x_1\in\mathcal{H}$ and for all $k\ge 1$
\begin{equation}\label{tdiscgen1}\tag{TIREPROG-p}
\left\{\begin{array}{lll}
y_k= x_k+\frac{((ak+1)^q-1)((ak+1-a)^q-1)}{(ak+1-a)^{2q}}(x_k-x_{k-1})-\frac{-(ak+1)^{2q}+(ak+1)^q+(ak+1-a)^{2q}}{(ak+1-a)^{2q}(ak+1)^{q}} x_k
\\
x_{k+1}=\prox_{sf}\left(y_k-s\n g(y_k)-s\frac{c}{k^p} y_k\right).
\end{array}\right.
\end{equation}

As a model result we present the  strong convergence results of the generated sequences to the minimum norm minimizer of $f+g$ and fast rates concerning the function values in the generated sequences, obtained for this particular choice of the parameters.

\begin{theorem}\label{tmodel} For $s<\frac{1}{L}$  let $(x_k)_{k\ge 0},\,(y_k)_{k\ge 1}$ be the sequences generated by Algorithm \eqref{tdiscgen1}.

\begin{enumerate}
\item[(i)]  If $\frac{1}{2}\le q<1$ and  $p<2q$ then $(x_k)$ converges strongly to $x^*$, where $\{x^*\}=\pr_{\argmin (f+g)}(0)$ is the minimum norm minimizer of our objective function $f+g.$ Moreover  $\| x_k - y_k \| \ = \ o\left( \sqrt{\e_k} \right) \mbox{ as } k \to +\infty,$ hence $(y_k)$ also converges strongly to $x^*.$
\\
Further, the following estimates hold.\\
$(f+g)(x_k)-\min (f+g)=\mathcal{O}\left(k^{-p}\right),\mbox{ as }k\to+\infty,$ $\| x_k - x_{k-1} \| \ = \ o\left(k^{-\frac{p}{2}} \right) \mbox{ as } k \to +\infty,$ and there exists $u_k\in\p f(x_k)+\n g(x_k)$ such that
$\| u_k \| = \ o\left( k^{-\frac{p}{2}}\right) \mbox{ as } k \to +\infty.$
\item[(ii)]  If $\frac{1}{2}\le q\le 1$ and  $p\ge2q$ and if $a=\frac12,\,q=1$ then $p>2,$ then
$(f+g)(x_k)-\min (f+g)=\mathcal{O}\left(k^{-2q}\right),\mbox{ as }k\to+\infty.$
\end{enumerate}
\end{theorem}

\begin{remark} Note that we cannot allow $q=1$ or $p\ge 2q$ in Theorem \ref{tmodel} (i). That is the reason why,  in concordance to the previously presented results concerning  continuous dynamical systems and algorithms,  neither Algorithm \eqref{tdiscgen1} will  provide the rate $(f+g)\left(x_k\right) - \min (f+g) =O\left(k^{-2}\right)$ as $k\to+\infty$ when we obtain the strong convergence of the generated sequences.
However, when $p$ is very close to $2$, then the convergence rate of the values  $(f+g)(x_k)-\min (f+g)=\mathcal{O}\left(k^{-p}\right),\mbox{ as }k\to+\infty,$ is  practically  as good as the rate $ (f+g)\left(x_k\right) - \min (f+g) =O\left(k^{-2}\right),\mbox{ as }k\to+\infty,$ obtained for Nesterov's algorithm or FISTA, and in contrast to the latter we obtain strong convergence of the generated sequences to the minimum norm minimizer of our objective function.

Nevertheless, according to (ii) we are able to obtain the rate $(f+g)\left(x_k\right) - \min (f+g) =O\left(k^{-2}\right)$ as $k\to+\infty$ in case $q=1$ and $p\ge 2.$ Unfortunately in this case concerning the iterates of \eqref{tdiscgen1} no convergence results are covered by our analysis.
\end{remark}

\subsection{The outline of the paper}

The paper is organized as follows. In the next section we prove the main result of the paper. We obtain strong convergence of the sequences generated by Algorithm \eqref{tdiscgen} and also fast convergence of the potential energy and discrete velocity. We give a special attention to the case when the optimal rate of order $\mathcal{O}(k^{-2})$ for the potential energy $(f+g)(x_k)-\min(f+g)$ can be obtained. In section 3 we consider the parameters in a simple form and discuss the conditions these parameters must satisfy in order to obtain strong convergence of the generated sequences to the minimum norm minimizer. We also discuss  the correlation between the parameters $t_k$ and $\e_k$. Further, in section 4 via some numerical experiments we show that Algorithm \eqref{tdiscgen} indeed assures the convergence of the generated sequences to a minimal norm solution and has a very good behavior compared to some well known algorithms from the literature. We also show that the use of two Tikhonov regularization terms in our algorithm is essential for obtaining the strong convergence of the generated sequences to the minimum norm minimizer of the objective function. Finally, we conclude our paper and we outline some perspectives.

\section{Strong convergence and fast rates}

Let us reformulate \eqref{tdiscgen} in a more convenable form. To this purpose, for $k\in\N$ consider the strongly convex function $g_k(x)=g(x)+\frac{\e_k}{2}\|x\|^2.$ Then, $\n g_k(x)=\n g(x)+\e_kx$, hence  \eqref{tdiscgen} can equivalently be written as:
$x_0,x_1\in\mathcal{H}$ and for all $k\ge 1$
\begin{equation}\label{tdiscgen2}\tag{TIREPROG-str}
\left\{\begin{array}{lll}
y_k= x_k+\frac{(t_{k}-1)(t_{k-1}-1)}{t_{k-1}^2}(x_k-x_{k-1})-\frac{-t_{k}^2+t_{k}+t_{k-1}^2}{t_{k-1}^2t_{k}} x_k
\\
x_{k+1}=\prox_{sf}(y_k-s\n g_k(y_k)).
\end{array}\right.
\end{equation}

Note that $\n g_k$ is also Lipschitz continuous, with Lipshitz constant $L_k=L+\e_k$ and taking into account that $\e_k$ is non-increasing and $\e_k\to 0$ as $k\to+\infty$ we have that for all $0<s<\frac{1}{L}$  there exists $ k_0\in\N$ such that $\e_{k}\le \frac{1}{s}-L$ for all $k\ge k_0$, hence the stepsize $s$ satisfies the following property.
$$(S)\,\,\,\,s\in \left(0,\frac{1}{L+\e_{k_0}}\right]\subseteq \left(0,\frac{1}{L+\e_k}\right]\subset\left(0,\frac{1}{L}\right),\mbox{ for all }k\ge k_0.$$

\begin{remark} If we apply Lemma \ref{modeBT}, (see Appendix), to the function $f+g_k$ and stepsize $s\in\left(0,\frac{1}{\e_k+L}\right]$  we get
\begin{align}\label{tf4}
(f+g_k)(\prox\nolimits_{sf}(y-s\n g_k(y))) &\le (f+g_k)(x)-\frac{1}{s}\left\<y-\prox\nolimits_{sf}(y-s\n g_k(y)),x-y\right\>\\
\nonumber&-\frac{1}{2s}\|\prox\nolimits_{sf}(y-s\n g_k(y))-y\|^2-\frac{s}{2}\|\n g_k(y)-\n g_k(x)\|^2,\forall x,y\in\mathcal{H}.
\end{align}
\end{remark}

The following general result holds.

\begin{theorem}\label{tcrate} For a sequence $(t_k)_{k\ge 0}$ satisfying \eqref{condT} and the stepsize $s<\frac{1}{L}$ let $(x_k)_{k\ge 0},\,(y_k)_{k\ge 1}$ be the sequences generated by Algorithm \eqref{tdiscgen}.

Let $k_0$ such that $s\in\left(0,\frac{1}{L+\e_{k_0}}\right]$.
Assume that   $\liminf_{k\to+\infty}\frac{-t_k^2+t_k+t_{k-1}^2}{t_{k}}>0$ and the sequence $(t_k^2\e_k)_{k\ge k_0-1}$ is increasing. Assume  further, that  $\lim_{k\to+\infty}t_k^2\e_k=+\infty$ and
$\lim_{k\to+\infty}\frac{ t_k(\e_k-\e_{k+1})}{\e_k} =0$.

Then, $(x_k)$ converges strongly to $x^*$, where $\{x^*\}=\pr_{\argmin (f+g)}(0)$ is the minimum norm minimizer of our objective function $f+g.$ Moreover  $\| x_k - y_k \| \ = \ o\left( \sqrt{\e_k} \right) \mbox{ as } k \to +\infty,$ hence $(y_k)$ also converges strongly to $x^*.$
\\
Further, the following estimates hold.\\
$F_k(x_{k})-F_{k}(\ol x_{k})=o(\e_k)\mbox{ as }k\to+\infty,$
where $F_k(x)=(f+g_k)(x)=f(x)+g(x)+\frac{\e_k}{2}\|x\|^2$ and $\ol x_k$ is the unique minimizer of $F_k.$
$$(f+g)(x_k)-\min (f+g)=\mathcal{O}\left(\e_k\right),\mbox{ as }k\to+\infty,$$
$$\| x_k - x_{k-1} \| \ = \ o\left( \sqrt{\e_k} \right) \mbox{ as } k \to +\infty,$$
and there exists $u_k\in\p f(x_k)+\n g(x_k)$ such that
$$\| u_k \| = \ o\left( \sqrt{\e_k} \right) \mbox{ as } k \to +\infty.$$
\end{theorem}

\begin{proof}
{\bf I. Lyapunov analysis} Assume  that $k\ge k_0$. Consider the energy functional
$$E_k=2st_{k-1}^2(F_k(x_{k})-F_k(\ol x_{k}))+\|\eta_k-x^*\|^2,\,k\ge k_0,$$
where $\eta_k=\frac{t_{k-1}^2}{t_k-1}y_k-\frac{t_{k-1}^2}{t_k}x_k,\,k\ge 1.$
We show that $E_k=o(t_k^2q_k)$ as $k\to+\infty.$

Indeed, since $s\in\left(0,\frac{1}{L+\e_{k_0}}\right],$ we get that $s$ satisfies (S), hence \eqref{tf4} can be used for every $k\ge k_0.$  We take $y=y_k,\,x=x_k$ in \eqref{tf4}  and we get
\begin{align}\label{tforp}
(f+g_k)(x_{k+1}) -(f+g_k)(x_k)&\le -\frac{1}{s}\left\<y_k-x_{k+1},x_k-y_k\right\>-\frac{1}{2s}\|x_{k+1}-y_k\|^2-\frac{s}{2}\|\n g_k(y_k)-\n g_k(x_k)\|^2.
\end{align}

Now we take $y=y_k,\,x=x^*$ in \eqref{tf4}  and  we get

\begin{align}\label{tforq}
(f+g_k)(x_{k+1}) - (f+g_k)(x^*)&\le-\frac{1}{s}\left\<y_k-x_{k+1},x^*-y_k\right\>-\frac{1}{2s}\|x_{k+1}-y_k\|^2-\frac{s}{2}\|\n g_k(y_k)-\n g_k(x^*)\|^2.
\end{align}


Consider the non-negative sequences $(p_k)_{k\ge 0},\,(q_k)_{k\ge 0}$ defined by $p_k=2s(t_k^2-t_k)$ and $q_k=2s t_k.$ We multiply \eqref{tforp} with $p_k$ and \eqref{tforq} with $q_k$ and add to get
\begin{align}\label{trightfirst1}
(p_k+q_k)F_k(x_{k+1})-&p_kF_k(x_{k})-q_kF_k(x^*)\le-\frac{s}{2}p_k\|\n g_k(y_k)-\n g_k(x_k)\|^2-\frac{s}{2}q_k\|\n g_k(y_k)-\n g_k(x^*)\|^2\\
\nonumber&+ \left\<\frac{1}{s}(y_k-x_{k+1}),(p_k+q_k)y_k-p_k x_k-q_k x^*-\frac{1}{2}(p_k+q_k)(y_k-x_{k+1})\right\>,
\end{align}
for all $k\ge k_0.$
In what follows we deal with the left-hand side of \eqref{trightfirst1}. By taking $x=x_{k+1},\,y=x_k,\,z=x^*$ and $z_k^*=\e_kx^*\in \p F_k(x^*)$ in Lemma \ref{forleftside} we get
\begin{align}\label{trightsecond}
&(p_{k}+q_{k})F_k(x_{k+1})-p_{k} F_k(x_k)-q_{k}F_k(x^*)\ge (p_{k}+q_{k})(F_{k+1}(x_{k+1})-F_{k+1}(\ol x_{k+1}))\\
\nonumber&-(p_{k-1}+q_{k-1})(F_k(x_k)-F_k(\ol x_{k}))+(p_{k-1}+q_{k-1}-p_{k})(F_k(x_k)-F_k(\ol x_{k}))\\
\nonumber&+(p_{k}+q_{k})\frac{\e_{k+1}-\e_k}{2}\|\ol x_{k+1}\|^2+q_{k}\e_k\<x^*,\ol x_{k+1}-x^*\>+p_k\frac{\e_{k}}{2}\|\ox_{k+1}-\ox_k\|^2+q_k\frac{\e_k}{2}\|\ox_{k+1}-x^*\|^2,
\end{align}
for all $k\ge k_0.$

Combining \eqref{trightfirst1} and \eqref{trightsecond} and taking into account the form of $p_k$ and $q_k$ we get

\begin{align}\label{trightthird}
&2st_k^2(F_{k+1}(x_{k+1})-F_{k+1}(\ol x_{k+1}))-2st_{k-1}^2(F_k(x_k)-F_k(\ol x_{k}))+2s(t_{k-1}^2-t_{k}^2+t_k)(F_k(x_k)-F_k(\ol x_{k}))\le\\
\nonumber&  \left\<y_k-x_{k+1},2t_k^2y_k-2(t_k^2-t_k) x_k-2t_k x^*-t_k^2(y_k-x_{k+1})\right\>-s^2(t_k^2-t_k)\|\n g_k(y_k)-\n g_k(x_k)\|^2\\
\nonumber&-s^2t_k\|\n g_k(y_k)-\n g_k(x^*)\|^2-s(t_k^2-t_k)\e_{k}\|\ox_{k+1}-\ox_k\|^2-s t_k\e_k\|\ox_{k+1}-x^*\|^2\\
\nonumber&+st_k^2(\e_k-\e_{k+1})\|\ol x_{k+1}\|^2+2st_k\e_k\<x^*,x^*-\ol x_{k+1}\>,
\end{align}
for all  $k\ge k_0.$

Now, according to Lemma \ref{foretandif} one has
\begin{align}\label{leftreldif1}
&
\left\<y_k-x_{k+1},2t_k^2y_k-2(t_k^2-t_k) x_k-2t_k x^*-t_k^2(y_k-x_{k+1})\right\>=\|\eta_k-x^*\|^2-\|\eta_{k+1}-x^*\|^2\\
\nonumber&-\frac{-t_k^2+t_k+t_{k-1}^2}{t_{k-1}^2}\|\eta_k-x^*\|^2-\frac{t_k^2-t_k}{t_{k-1}^2}\cdot\frac{-t_k^2+t_k+t_{k-1}^2}{t_{k-1}^2}\|\eta_k\|^2+\frac{-t_k^2+t_k+t_{k-1}^2}{t_{k-1}^2}\|x^*\|^2,
\end{align}
for all  $k\ge k_0.$ 
By injecting \eqref{leftreldif1} in \eqref{trightthird} yields

\begin{align}\label{trightthird1}
&(2st_k^2(F_{k+1}(x_{k+1})-F_{k+1}(\ol x_{k+1}))+\|\eta_{k+1}-x^*\|^2)-(2st_{k-1}^2(F_k(x_k)-F_k(\ol x_{k}))+\|\eta_k-x^*\|^2)\\
\nonumber&+\frac{t_{k-1}^2-t_{k}^2+t_k}{t_{k-1}^2}(2st_{k-1}^2(F_k(x_k)-F_k(\ol x_{k}))+\|\eta_k-x^*\|^2)\le\\
\nonumber& -s^2(t_k^2-t_k)\|\n g_k(y_k)-\n g_k(x_k)\|^2-s^2t_k\|\n g_k(y_k)-\n g_k(x^*)\|^2\\
\nonumber&-s(t_k^2-t_k)\e_{k}\|\ox_{k+1}-\ox_k\|^2-s t_k\e_k\|\ox_{k+1}-x^*\|^2-\frac{t_k^2-t_k}{t_{k-1}^2}\cdot\frac{-t_k^2+t_k+t_{k-1}^2}{t_{k-1}^2}\|\eta_k\|^2\\
\nonumber&+\frac{-t_k^2+t_k+t_{k-1}^2}{t_{k-1}^2}\|x^*\|^2+st_k^2(\e_k-\e_{k+1})\|\ol x_{k+1}\|^2+2st_k\e_k\<x^*,x^*-\ol x_{k+1}\>,
\end{align}
for all  $k\ge k_0.$

By neglecting the non-positive terms $-s^2(t_k^2-t_k)\|\n g_k(y_k)-\n g_k(x_k)\|^2$, $-s^2t_k\|\n g_k(y_k)-\n g_k(x^*)\|^2$,
$-s(t_k^2-t_k)\e_{k}\|\ox_{k+1}-\ox_k\|^2$ and $-\frac{t_k^2-t_k}{t_{k-1}^2}\cdot\frac{-t_k^2+t_k+t_{k-1}^2}{t_{k-1}^2}\|\eta_k\|^2$ we get

\begin{align}\label{trightthird2}
&(2st_k^2(F_{k+1}(x_{k+1})-F_{k+1}(\ol x_{k+1}))+\|\eta_{k+1}-x^*\|^2)-(2st_{k-1}^2(F_k(x_k)-F_k(\ol x_{k}))+\|\eta_k-x^*\|^2)\\
\nonumber&+\frac{-t_{k}^2+t_k+t_{k-1}^2}{t_{k-1}^2}(2st_{k-1}^2(F_k(x_k)-F_k(\ol x_{k}))+\|\eta_k-x^*\|^2)\le\\
\nonumber&-s t_k\e_k\|\ox_{k+1}-x^*\|^2+\frac{-t_k^2+t_k+t_{k-1}^2}{t_{k-1}^2}\|x^*\|^2+st_k^2(\e_k-\e_{k+1})\|\ol x_{k+1}\|^2+2st_k\e_k\<x^*,x^*-\ol x_{k+1}\>,
\end{align}
for all  $k\ge k_0.$

Now, by denoting $\a_k= \frac{-t_{k}^2+t_k+t_{k-1}^2}{t_{k-1}^2}$ as in the proof of Lemma \ref{foretandif}, in terms of the energy functional $E_k=2st_{k-1}^2(F_k(x_{k})-F_k(\ol x_{k}))+\|\eta_k-x^*\|^2,$ \eqref{trightthird2} reads as
\begin{align}\label{energy}
E_{k+1}-E_k+\a_kE_k&\le-s t_k\e_k\|\ox_{k+1}-x^*\|^2+\a_k\|x^*\|^2+st_k^2(\e_k-\e_{k+1})\|\ol x_{k+1}\|^2+2st_k\e_k\<x^*,x^*-\ol x_{k+1}\>\\
\nonumber&=\a_k\|x^*\|^2+st_k^2(\e_k-\e_{k+1})\|\ol x_{k+1}\|^2+st_k\e_k(\|x^*\|^2-\|\ol x_{k+1}\|^2)
\end{align}
for all  $k\ge k_0.$

{\bf II. The rate of $E_k$.} Consider now the sequence $\pi_k=\frac{1}{\prod_{i=k_0}^{k}(1-\a_i)}$ and observe that since $\a_i<1$ for all $i\ge k_0$ we have that $(\pi_k)_{k\ge k_0}$ is an increasing sequence.

Let us multiply \eqref{energy} with $\pi_k,\,k>k_0$ to get
\begin{equation}\label{energy1}
\pi_kE_{k+1}-\pi_{k-1}E_k\le \pi_k \a_k\|x^*\|^2+st_k^2(\e_k-\e_{k+1})\pi_k\|\ol x_{k+1}\|^2+st_k\e_k\pi_k(\|x^*\|^2-\|\ol x_{k+1}\|^2),\mbox{ for all }k>k_0.
\end{equation}

By summing up \eqref{energy1} from $k=k_0+1$ to $k=n>k_0+1$ we obtain
 \begin{align}\label{energy2}
\pi_n E_{n+1}&\le\sum_{k=k_0+1}^n \pi_k\a_k\|x^*\|^2+\sum_{k=k_0+1}^nst_k^2(\e_k-\e_{k+1})\pi_k\|\ol x_{k+1}\|^2+\sum_{k=k_0+1}^n st_k\e_k\pi_k(\|x^*\|^2-\|\ol x_{k+1}\|^2)\\
\nonumber&+\pi_{k_0}E_{k_0}.
\end{align}
Note that $\pi_k\a_k=\pi_k-\pi_{k-1},$ hence
$\sum_{k=k_0+1}^n \pi_k\a_k\|x^*\|^2=\pi_n\|x^*\|^2-\pi_{k_0}\|x^*\|^2.$
Consequently, \eqref{energy2} is equivalent to
 \begin{align}\label{energy3}
 E_{n+1}&\le\|x^*\|^2+\frac{\sum_{k=k_0+1}^nst_k^2(\e_k-\e_{k+1})\pi_k\|\ol x_{k+1}\|^2}{\pi_n}+\frac{\sum_{k=k_0+1}^n st_k\e_k\pi_k(\|x^*\|^2-\|\ol x_{k+1}\|^2)}{\pi_n}+\frac{C}{\pi_n},
\end{align}
where  $C=\pi_{k_0}E_{k_0}-\pi_{k_0}\|x^*\|^2.$

 Next we show that the right hand side of \eqref{energy3} is of order $o(t_n^2\e_n)\mbox{ as }n\to+\infty.$

Indeed, according to the hypotheses $t_n^2\e_n\to+\infty$ as $n\to+\infty$ and we know that $(\pi_n)$ is increasing, hence $\frac{\pi_n\|x^*\|+C}{\pi_n}=o(t_n^2\e_n)\mbox{ as }n\to+\infty.$

Further, since  according to the hypotheses  $\liminf_{n\to+\infty}\frac{t_{n-1}^2}{t_n}\a_n=\liminf_{n\to+\infty}\frac{-t_{n}^2+t_n+t_{n-1}^2}{t_{n}}>0$, the sequence  $(t_n^2\e_n\pi_n)$ is increasing  and $\lim_{n\to+\infty}t_n^2\e_n\pi_n=+\infty$,  by using the fact that $\lim_{n\to+\infty}(\|x^*\|^2-\|\ol x_{n+1}\|^2)=0$, via the Ces\`aro-Stolz theorem  we get
\begin{align*}
&\lim_{n\to+\infty}\frac{\sum_{k=k_0+1}^n st_k\e_k\pi_k(\|x^*\|^2-\|\ol x_{k+1}\|^2)}{t_n^2\e_n\pi_n}=s\lim_{n\to+\infty}\frac{t_n\e_n\pi_n (\|x^*\|^2-\|\ol x_{n+1}\|^2)}{t_n^2\e_n\pi_n-t_{n-1}^2\e_{n-1}\pi_{n-1}}\\
&=s\lim_{n\to+\infty}\frac{ \|x^*\|^2-\|\ol x_{n+1}\|^2}{t_n-\frac{t_{n-1}^2\e_{n-1}\pi_{n}(1-\a_n)}{t_n\e_n\pi_n}}=s\lim_{n\to+\infty}\frac{\|x^*\|^2-\|\ol x_{n+1}\|^2}{\frac{t_n^2\e_n-t_{n-1}^2\e_{n-1}}{t_n\e_n}+\a_n\frac{t^2_{n-1}\e_{n-1}}{t_n\e_n}}\le s\lim_{n\to+\infty}\frac{\|x^*\|^2-\|\ol x_{n+1}\|^2}{\a_n\frac{t^2_{n-1}\e_{n-1}}{t_n\e_n}} =0.
\end{align*}

Finally, according to the hypotheses $\lim_{n\to+\infty}\frac{ t_n(\e_n-\e_{n+1})}{\e_n} =0$, hence for some $M>0$ one has
\begin{align*}&\lim_{n\to+\infty}\frac{\sum_{k=k_0+1}^nst_k^2(\e_k-\e_{k+1})\pi_k\|\ol x_{k+1}\|^2}{t_n^2\e_n\pi_n}=
s\lim_{n\to+\infty}\frac{t_n^2(\e_n-\e_{n+1})\pi_n\|\ol x_{n+1}\|^2}{t_n^2\e_n\pi_n-t_{n-1}^2\e_{n-1}\pi_{n-1}}\\
&=s\lim_{n\to+\infty}\frac{ (\e_n-\e_{n+1})\|\ol x_{n+1}\|^2}{\e_n-\frac{t_{n-1}^2\e_{n-1}\pi_{n-1}}{t_n^2\pi_n}}=s\lim_{n\to+\infty}\frac{(\e_n-\e_{n+1})\|\ol x_{n+1}\|^2}{\frac{t_n^2\e_n-t_{n-1}^2\e_{n-1}}{t_n^2}+\a_n\frac{t^2_{n-1}\e_{n-1}}{t_n^2}}\\
&=s\lim_{n\to+\infty}\frac{\|\ol x_{n+1}\|^2}{\frac{t_n^2\e_n-t_{n-1}^2\e_{n-1}}{t_n\e_n}+\a_n\frac{t^2_{n-1}\e_{n-1}}{t_n\e_n}}\cdot\frac{t_n(\e_n-\e_{n+1})}{\e_n}\le M\lim_{n\to+\infty}\frac{ t_n(\e_n-\e_{n+1})}{\e_n} =0.
\end{align*}
Consequently, from \eqref{energy3} we get
\begin{equation}\label{energyrate}
E_{n+1}=o(t_n^2\e_n)\mbox{ as }n\to+\infty.
\end{equation}

{\bf III. Strong Convergence and Rates}

 Now using \eqref{energyrate} we derive the strong convergence of $(x_n)$ and $(y_n)$ to $x^*$ and the convergence rates concerning the potential energies, discrete velocity and values of the gradient stated in the conclusion of the theorem.

Indeed, taking into account the form of $E_n$, \eqref{energyrate} leads to
$(2st_{n}^2(F_{n+1}(x_{n+1})-F_{n+1}(\ol x_{n+1}))=o(t_n^2\e_n)\mbox{ as }n\to+\infty$. In other words,
$F_{n}(x_{n})-F_{n}(\ox_{n})=o(\e_n)\mbox{ as }n\to+\infty$ and by using \eqref{fontos5}, that is
$$(f+g)(x_n)-(f+g)(x^*)\le F_n(x_n)-F_n(\ox_n)+\frac{\e_n}{2}\|x^*\|^2,$$
 we get
$$(f+g)(x_n)-\min (f+g)=\mathcal{O}(\e_n)\mbox{ as }n\to+\infty.$$

In order to show strong convergence, we use \eqref{fontos3}, that is
$$F_n(x_n)-F_n(\ox_n)\ge\frac{\e_n}{2}\|x_n-\ox_n\|^2,$$
 and we get
\begin{align*}\lim_{n\to+\infty}\|x_n-x^*\|^2&\le 2\lim_{n\to+\infty}\left(\|x_n-\ol x_n\|^2+\|\ol x_n-x^*\|^2\right)\\
&\le4 \lim_{n\to+\infty}\frac{F_n(x_n)-F_n(\ol x_n)}{\e_n}+2\lim_{n\to+\infty}\|\ol x_n-x^*\|^2=0.
\end{align*}

Concerning the convergence rates  for the discrete velocity $\| x_n - x_{n-1} \|$ we conclude the following. From the definition of $E_n$ and the fact that $E_n=o\left(t_n^2 \e_n \right) \mbox{ as } n \to +\infty$
we have that
$$\| \eta_n - x^* \| = o\left( t_n \sqrt{\e_n} \right) \mbox{ as } n \to +\infty.$$
Now, using the definition of $\eta_n$ and the form of $y_n$ we derive

\begin{align*}
    \eta_n -x^*\ &= \frac{t_{n-1}^2}{t_n-1}\left(x_n+\frac{(t_{n}-1)(t_{n-1}-1)}{t_{n-1}^2}(x_n-x_{n-1})-\frac{-t_{n}^2+t_{n}+t_{n-1}^2}{t_{n-1}^2t_{n}} x_n\right)-\frac{t_{n-1}^2}{t_n}x_n-x^*\\
    &= (t_{n-1}-1) (x_n - x_{n-1}) + x_n-x^*,
\end{align*}
hence $(t_{n-1}-1) \|x_n - x_{n-1}\|=o\left( t_n \sqrt{\e_n} \right) \mbox{ as } n \to +\infty.$
Now, since $(t_n^2\e_n)$ is increasing, and $(\e_n)$ is non-increasing we deduce that $(t_n)$ is increasing. Further, since $\lim_{n\to+\infty} t_n^2\e_n=+\infty$ and $\lim_{n\to+\infty} \e_n=0$ we obtain that $\lim_{n\to+\infty} t_n=+\infty.$ Consequently,
$\lim_{n\to+\infty}\frac{t_{n-1}-1}{t_{n}}=1.$

Therefore,
$$ \|x_n - x_{n-1}\|=o\left(\sqrt{\e_n} \right) \mbox{ as } n \to +\infty.$$

From here and the fact that $y_n=x_n+\frac{(t_{n}-1)(t_{n-1}-1)}{t_{n-1}^2}(x_n-x_{n-1})-\frac{-t_{n}^2+t_{n}+t_{n-1}^2}{t_{n-1}^2t_{n}} x_n$, we deduce at once that $\|y_n-x_n\|=o(\sqrt{\e_n})$ as $n\to+\infty,$ hence in particular
$$\lim_{n\to+\infty}y_n=x^*.$$

Let us show the estimate concerning the sub-gradients, that is, there exists $u_n\in\p f(x_n)+\n g(x_n)$ such that
$$\| u_n \| = \ o\left( \sqrt{\e_n} \right) \mbox{ as } n \to +\infty.$$  To this purpose, we reformulate \eqref{tdiscgen} in terms of the resolvent operator of $\p f,$ i.e.
$$x_{n+1}+s\p f(x_{n+1})\ni y_n-s\n g(y_n)-s\e_n y_n.$$
In other words, there exists $u_{n+1}\in\p f(x_{n+1})+\n g(x_{n+1})$ such that
$$su_{n+1}= (y_n-x_{n+1})+s(\n g(x_{n+1})-\n g(y_n))-s\e_n y_n.$$
Now, by using the $L-$Lipschitz continuity of $\n g$ we get
$$\|\n g(x_{n+1})-\n g(y_n)\|\le\|\|\n g(y_n)-\n g(x_{n})\|+\|\n g(x_n)-\n g(x_{n+1})\|\le L\|y_n-x_n\|+L\|x_{n+1}-x_n\|,$$
hence $\|\n g(y_n)-\n g(x_{n+1})\|=o(\sqrt{\e_n})$ as $n\to+\infty.$ Further, $y_n-x_{n+1}=(y_n-x_n)+(x_n-x_{n+1})$, hence
$\|y_n-x_{n+1}\|=o(\sqrt{\e_n})$ as $n\to+\infty.$ Consequently, $\|(y_n-x_{n+1})+s(\n g(x_{n+1})-\n g(y_n))-s\e_n y_n\|=o(\sqrt{\e_n})$ as $n\to+\infty,$ that is,
$$\| u_{n+1} \| = \ o\left( \sqrt{\e_n} \right) \mbox{ as } n \to +\infty.$$
\end{proof}

\begin{remark}\label{afterstrconv} According to the hypotheses of Theorem \ref{tcrate} in order to obtain strong convergence of the sequences generated by \eqref{tdiscgen} we need to assume that $$\lim_{n\to+\infty} t_n^2\e_n=+\infty.$$ Beside the strong convergence of the generated sequences we obtained $(f+g)(x_n)-\min(f+g)=\mathcal{O}(\e_n)$ as $n\to+\infty.$ As we have seen in the proof of Theorem \ref{tcrate}, the above mentioned assumption leads to $\lim_{n\to+\infty }t_n=+\infty$ and $\lim_{n\to+\infty}\frac{t_{n-1}-1}{t_n}=1.$  Consequently, the assumption $\lim_{n\to+\infty} t_n^2\e_n=+\infty$ implies that the inertial parameter in Algorithm \eqref{tdiscgen}, that is $\frac{(t_{n}-1)(t_{n-1}-1)}{t_{n-1}^2}$ goes to $1$ as $n\to+\infty$ just as in FISTA. Further, the parameter $\frac{-t_{n}^2+t_{n}+t_{n-1}^2}{t_{n-1}^2t_{n}}$ is indeed a Tikhonov regularization parameter, since is non-negative and goes to $0$ as $n\to+\infty.$

However,  in case $$\limsup_{n\to+\infty}t_n^2\e_n<+\infty$$ even faster rates can be obtained for the potential energy $(f+g)(x_n)-\min(f+g).$ More precisely, if the right hand side of \eqref{energy3} is bounded from above, then  there exists $M>0$ such that

$$ E_{n+1}\le M,\mbox{ for all }n\ge k_0+1.$$
The latter relation may assure the rate $(f+g)(x_n)-\min(f+g)=\mathcal{O}\left(\frac{1}{t_n^2}\right)$ as $n\to+\infty$ just as the case of FISTA. Unfortunately, in that case we cannot obtain any convergence result concerning the sequences generated by \eqref{tdiscgen}.

The following result holds.
\end{remark}
\begin{theorem}\label{fastrates}
For a sequence $(t_k)_{k\ge 0}$ satisfying \eqref{condT} and the stepsize $s<\frac{1}{L}$ let $(x_k)_{k\ge 0}$ be the sequence generated by Algorithm \eqref{tdiscgen}.

Let $k_0\in\N$ such that $s\in\left(0,\frac{1}{L+\e_{k_0}}\right]$ and consider the sequences $\a_k=\frac{-t_k^2+t_k+t_{k-1}^2}{t_{k-1}^2}$ and $\pi_k=\frac{1}{\prod_{i=k_0}^k(1-\a_i)}$ as in the proof of Theorem \ref{tcrate}.

 Assume further that one of the set of the following conditions holds.
 \begin{enumerate}
\item[(a)] $\lim_{k\to+\infty}\pi_k=+\infty$, $\limsup_{k\to+\infty}\frac{t_k^2(\e_k-\e_{k+1})}{\a_k}<+\infty$ and $\limsup_{k\to+\infty}\frac{t_k\e_k}{\a_k}<+\infty;$
\item[(b)] $\lim_{k\to+\infty}\pi_k<+\infty$, $\sum_{k=1}^{+\infty}t_k^2(\e_k-\e_{k+1})<+\infty$ and $\sum_{k=1}^{+\infty}t_k\e_k<+\infty.$
 \end{enumerate}

Then,
$F_k(x_{k})-F_{k}(\ol x_{k})=\mathcal{O}\left(\frac{1}{t_{k-1}^2}\right)\mbox{ as }k\to+\infty,$
where $F_k(x)=(f+g_k)(x)=f(x)+g(x)+\frac{\e_k}{2}\|x\|^2$ and $\ol x_k$ is the unique minimizer of $F_k,$
and
$$(f+g)(x_k)-\min (f+g)=\mathcal{O}\left(\frac{1}{t_{k-1}^2}\right)+\mathcal{O}\left(\e_k\right),\mbox{ as }k\to+\infty.$$

\end{theorem}
\begin{proof}
We use the notations from the proof of Theorem \ref{tcrate}. Then \eqref{energy3} gives us
 \begin{align}\label{interm}
 E_{n+1}&\le\|x^*\|^2+\frac{\sum_{k=k_0+1}^nst_k^2(\e_k-\e_{k+1})\pi_k\|\ol x_{k+1}\|^2}{\pi_n}+\frac{\sum_{k=k_0+1}^n st_k\e_k\pi_k(\|x^*\|^2-\|\ol x_{k+1}\|^2)}{\pi_n}+\frac{C}{\pi_n},
\end{align}
for all $n> k_0+1$, where  $C=\pi_{k_0}E_{k_0}-\pi_{k_0}\|x^*\|^2.$

Note that if $(b)$ holds, then from \eqref{interm} we get that there exists $M>0$ such that
$E_{n+1}\le M$ consequently, by taking into account the form of $E_{n+1}$, we get
$F_{n+1}(x_{n+1})-F_{n+1}(\ol x_{n+1})\le \frac{M}{2s t_n^2}$, hence

$$F_n(x_{n})-F_{n}(\ol x_{n})=\mathcal{O}\left(\frac{1}{t_{n-1}^2}\right)\mbox{ as }n\to+\infty.$$
Now, by using \eqref{fontos5}, that is
$$(f+g)(x_n)-(f+g)(x^*)\le F_n(x_n)-F_n(\ox_n)+\frac{\e_n}{2}\|x^*\|^2,$$
 we get
$$(f+g)(x_n)-\min (f+g)=\mathcal{O}\left(\frac{1}{t_{n-1}^2}\right)+\mathcal{O}(\e_n)\mbox{ as }n\to+\infty.$$

It remained to show that if case $(a)$ holds, then the right hand side of \eqref{interm} is bounded. Note that $(\pi_n)_{n\ge k_0}$ is increasing and due to our assumption  we have $\lim_{n\to+\infty}\pi_n=+\infty$, further $\pi_n-\pi_{n-1}=\pi_n\a_n,$ hence, by Ces\`aro-Stolz theorem  we get

$$\lim_{n\to+\infty}\frac{\sum_{k=k_0+1}^nst_k^2(\e_k-\e_{k+1})\pi_k\|\ol x_{k+1}\|^2}{\pi_n}=\lim_{n\to+\infty}\frac{st_n^2(\e_n-\e_{n+1})\pi_n\|\ol x_{n+1}\|^2}{\pi_n\a_n}<+\infty$$
and
$$\lim_{n\to+\infty}\frac{\sum_{k=k_0+1}^n st_k\e_k\pi_k(\|x^*\|^2-\|\ol x_{k+1}\|^2)}{\pi_n}=\lim_{n\to+\infty}\frac{st_n\e_n\pi_n(\|x^*\|^2-\|\ol x_{n+1}\|^2)}{\pi_n\a_n}<+\infty.$$
Consequently, there exists $M>0$ such that
$E_{n+1}\le M$ and the rest of the proof goes analogously as in the case $(b).$

\end{proof}

\section{Particular choices of $t_k$ and $\e_k$.}

In this section we consider some particular choices of $t_k$ and $\e_k$ in Theorem \ref{tcrate} and Theorem \ref{fastrates}. More precisely we take $t_k=(ak+1)^q,\,k\ge 0$ with $\frac{1}{2}\le q\le 1$ and $0<a\le\frac{1}{2q}$ and let $\e_k=\frac{c}{k^p},\,c,p>0.$ As we seen in section \ref{onT}, in this case $t_k$ satisfies \eqref{condT} and \eqref{tdiscgen} has the form:
$x_0,x_1\in\mathcal{H}$ and for all $k\ge 1$
\begin{equation}\label{tdiscgen11}\tag{TIREPROG-p}
\left\{\begin{array}{lll}
y_k= x_k+\frac{((ak+1)^q-1)((ak+1-a)^q-1)}{(ak+1-a)^{2q}}(x_k-x_{k-1})-\frac{-(ak+1)^{2q}+(ak+1)^q+(ak+1-a)^{2q}}{(ak+1-a)^{2q}(ak+1)^{q}} x_k
\\
x_{k+1}=\prox_{sf}\left(y_k-s\n g(y_k)-s\frac{c}{k^p} y_k\right).
\end{array}\right.
\end{equation}

Concerning the strong convergence of the iterates of \eqref{tdiscgen11} the following result holds.

\begin{theorem}\label{tcratep} Assume that $\frac12\le q<1$ and $p<2q.$ For  the stepsize $s<\frac{1}{L}$ let $(x_k)_{k\ge 0},\,(y_k)_{k\ge 1}$ be the sequences generated by Algorithm \eqref{tdiscgen11}.

Then, $(x_k)$ converges strongly to $x^*$, where $\{x^*\}=\pr_{\argmin (f+g)}(0)$ is the minimum norm minimizer of our objective function $f+g.$ Moreover  $\| x_k - y_k \| \ = \ o\left( \sqrt{\e_k} \right) \mbox{ as } k \to +\infty,$ hence $(y_k)$ also converges strongly to $x^*.$
\\
Further, the following estimates hold.\\
$F_k(x_{k})-F_{k}(\ol x_{k})=o(k^{-p})\mbox{ as }k\to+\infty,$
where $F_k(x)=(f+g_k)(x)=f(x)+g(x)+\frac{c}{2k^p}\|x\|^2$ and $\ol x_k$ is the unique minimizer of $F_k.$
$$(f+g)(x_k)-\min (f+g)=\mathcal{O}\left(k^{-p}\right),\mbox{ as }k\to+\infty,$$
$$\| x_k - x_{k-1} \| \ = \ o\left( k^{-\frac{p}{2}} \right) \mbox{ as } k \to +\infty,$$
and there exists $u_k\in\p f(x_k)+\n g(x_k)$ such that
$$\| u_k \| = \ o\left( k^{-\frac{p}{2}}  \right) \mbox{ as } k \to +\infty.$$
\end{theorem}
\begin{proof}
We just need to show that the conditions from the hypotheses of Theorem \ref{tcrate} hold.

First, let $\ol k=int\left(\frac{cs}{1-sL}\right)+1,$ where $int(x)$ denotes the integer part of $x$.   Then $s\in\left(0,\frac{1}{L+\e_{k}}\right]$ for all $k\ge\ol k.$ Consider now $k_0=\max\left(\ol k,int\left(\frac{p}{a(2q-p)}\right)+2\right).$
It remaind to show that
 \begin{enumerate}
 \item $\liminf_{k\to+\infty}\frac{-t_k^2+t_k+t_{k-1}^2}{t_{k}}>0.$
 \item $(t_k^2\e_k)_{k\ge k_0-1}$ is increasing.
 \item  $\lim_{k\to+\infty}t_k^2\e_k=+\infty$.
  \item $\lim_{k\to+\infty}\frac{ t_k(\e_k-\e_{k+1})}{\e_k} =0$.
\end{enumerate}

Note that since $q<1$ one has
$$\liminf_{k\to+\infty}\frac{-t_k^2+t_k+t_{k-1}^2}{t_{k}}=\lim_{k\to+\infty}\left(1+\frac{(ak+1-a)^{2q}-(ak+1)^{2q}}{(ak+1)^q}\right)=1.$$

Further, since $p<2q$ one has
$$t_k^2\e_k=c(ak+1)^{2q}k^{-p}$$
is an increasing sequence for all $k\ge int\left(\frac{p}{a(2q-p)}\right)+1.$ To see this just consider the real  valued real function
$\phi(x)=c(ax+1)^{2q}x^{-p}$ and observe that $\phi'(x)=c(ax+1)^{2q-1}x^{-p-1}(2qa x-p(ax+1))>0$ if $x>\frac{p}{a(2q-p)}.$

Moreover, since $p<2q$ one has $\lim_{k\to+\infty}t_k^2\e_k=\lim_{k\to+\infty}\frac{c(ak+1)^{2q}}{k^p}=+\infty.$

Finally, since $q<1$ one has $\lim_{k\to+\infty}\frac{ t_k(\e_k-\e_{k+1})}{\e_k} =\lim_{k\to+\infty}\frac{c(ak+1)^q}{k+1}\frac{(k+1)^p-k^p}{(k+1)^{p-1}}=0.$
\end{proof}

Concerning the case $p\ge 2q$ the following result holds.

\begin{theorem}\label{fastratesp}
Assume that $\frac12\le q\le1$ and $p\ge2q$ and in case $q=1,\,a=\frac12$ one has $p>2.$ For  the stepsize $s<\frac{1}{L}$ let $(x_k)_{k\ge 0}$ be the sequence generated by Algorithm \eqref{tdiscgen11}.

Then,
$$(f+g)(x_k)-\min (f+g)=\mathcal{O}\left(k^{-2q}\right),\mbox{ as }k\to+\infty.$$

\end{theorem}
\begin{proof}
Let $k_0\in\N$ such that $s\in\left(0,\frac{1}{L+\e_{k_0}}\right]$ and consider the sequences
$$\a_k=\frac{-t_k^2+t_k+t_{k-1}^2}{t_{k-1}^2}=\frac{(ak+1-a)^{2q}-(ak+1)^{2q}+(ak+1)^{q}}{(ak+1-a)^{2q}},\,k\ge k_0$$ and $$\pi_k=\frac{1}{\prod_{i=k_0}^k(1-\a_i)}=\frac{(a k_0+1-a)^q}{(ak+1)^q}\prod_{i=k_0}^k\frac{(ai+1-a)^q}{(ai+1)^q-1}.$$

Note that if $q<1$ then $\a_k=\mathcal{O}\left(\frac{1}{k^q}\right)$ as $k\to+\infty.$ More precisely, there exists $C_1,C_2>0$ such that $\frac{C_1}{k^q}\le \a_k\le\frac{C_2}{k^q}$ for all $k\ge \ol k$ and $\ol k$ big enough. Consequently, by using the facts that
$$\frac{1}{\prod_{i=k_0}^{+\infty}(1-\a_i)}\ge \prod_{i=k_0}^{+\infty}(1+\a_i)\ge \sum_{i=k_0}^{+\infty}\a_i\ge \sum_{i=k_0}^{\ol k-1}\a_i+\sum_{i=\ol k}^{+\infty}\frac{C_1}{i^q}=+\infty.$$

Hence, for $q<1$ one has $\lim_{k\to+\infty}\pi_k=+\infty$ and we have to show that all the assumptions (a)  in the hypotheses of Theorem \ref{fastrates} hold. More precisely, we need to show that $\limsup_{k\to+\infty}\frac{t_k^2(\e_k-\e_{k+1})}{\a_k}<+\infty$ and $\limsup_{k\to+\infty}\frac{t_k\e_k}{\a_k}<+\infty.$

But $\e_k-\e_{k+1}=\mathcal{O}(k^{-p-1})$ as $k\to+\infty$, hence $\frac{t_k^2(\e_k-\e_{k+1})}{\a_k}=\mathcal{O}(k^{3q-p-1})$ as $k\to+\infty$ and consequently
$\limsup_{k\to+\infty}\frac{t_k^2(\e_k-\e_{k+1})}{\a_k}<+\infty$ provided $3q-p-1\le 0$. But due to our assumption $2q\le p$ and $q<1$ hence $3q-1<2q\le p$ and the conclusion follows.

Further, $\frac{t_k\e_k}{\a_k}=\mathcal{O}(k^{2q-p})$ as $k\to+\infty$, consequently
$\limsup_{k\to+\infty}\frac{t_k\e_k}{\a_k}<+\infty.$

Now, in case $q=1$ we get
$$\a_k=\frac{(ak+1-a)^{2}-(ak+1)^{2}+ak+1}{(ak+1-a)^{2}}=\frac{(a-2a^2)k+(a-1)^2}{(ak+1-a)^{2}},\,k\ge k_0$$
therefore if $a<\frac12$ one has $\a_k=\mathcal{O}(k^{-1})$ as  $k\to+\infty$ and when $a=\frac12$ one has $\a_k=\frac{1}{(k+1)^2}.$

For $a<\frac12$ proceeding as before we get
$$\frac{1}{\prod_{i=k_0}^{+\infty}(1-\a_i)}\ge \prod_{i=k_0}^{+\infty}(1+\a_i)\ge \sum_{i=k_0}^{+\infty}\a_i=+\infty,$$
hence,  $\lim_{k\to+\infty}\pi_k=+\infty$. Further,  we have $\frac{t_k^2(\e_k-\e_{k+1})}{\a_k}=\mathcal{O}(k^{2-p})$ as $k\to+\infty,$
and using the fact that $p\ge 2$ we  get
$\limsup_{k\to+\infty}\frac{t_k^2(\e_k-\e_{k+1})}{\a_k}<+\infty.$
Similarly, $\frac{t_k\e_k}{\a_k}=\mathcal{O}(k^{2-p})$ as $k\to+\infty,$ hence
$\limsup_{k\to+\infty}\frac{t_k\e_k}{\a_k}<+\infty.$

Consequently, for $a<\frac12$ and $q=1$ the conditions assumed in (a) in the hypotheses of Theorem \ref{fastrates} hold.

For $a=\frac12$ we get
$$\lim_{k\to+\infty}\pi_k=\frac{1}{\prod_{i=k_0}^{+\infty}(1-\a_i)}=\prod_{i=k_0}^{+\infty}\frac{(i+1)^2}{i(i+2)}=\lim_{k\to+\infty}\frac{k+1}{k_0}\cdot\frac{k_0+1}{k+2}<+\infty.$$

Let us show that the other assumptions stated at (b) in the hypotheses of Theorem \ref{fastrates} hold, that is $\sum_{k=1}^{+\infty}t_k^2(\e_k-\e_{k+1})<+\infty$ and $\sum_{k=1}^{+\infty}t_k\e_k<+\infty.$

Note that since $t_k^2(\e_k-\e_{k+1}),\,t_k\e_k=\mathcal{O}(k^{1-p})$ as  $k\to+\infty$ and $p>2$ the claims follow.

\end{proof}

\begin{remark} Note that in case $q=1$ we are able to obtain the rate $\mathcal{O}(k^{-2})$ for the potential energy $(f+g)(x_k)-\min (f+g),$ which makes \eqref{tdiscgen1} comparable with the famous FISTA algorithm.
The optimal rate is attained  for all $p\ge 2$, provided $a<\frac12.$ The case $a=\frac12$ is special in the sense that in order to obtain the optimal rate one must assume that $p>2.$ Nevertheless, in case $a=\frac12$ our  algorithm \eqref{tdiscgen1} has the simple form:

\begin{equation}\label{tdiscgen111}\tag{TIREPROG-spec}
\left\{\begin{array}{lll}
x_0,x_1\in\mathcal{H},\mbox{ for }k\ge 1\\
y_k= x_k+\frac{(k-1)k}{(k+1)^2}(x_k-x_{k-1})-\frac{2}{(k+1)^2(k+2)}x_k
\\
x_{k+1}=\prox_{sf}\left(y_k-s\n g(y_k)-s\frac{c}{k^p} y_k\right).
\end{array}\right.
\end{equation}
We emphasize that \eqref{tdiscgen111} can easily be implemented and may outperform FISTA as some numerical experiments show.
\end{remark}

\section{Numerical experiments}
In this section we consider some numerical experiments in order to sustain the theoretical results obtained in the paper.

 In our first experiment we analyze the speed of convergence of the potential energy $(f+g)(x_n)-\min(f+g)$, for different choices of the inertial parameter and the Tikhonov regularization parameters in \eqref{tdiscgen1} and  we show that indeed algorithm \eqref{tdiscgen1} has a superior behaviour compared to the famous FISTA algorithm.

 To this purpose, let us consider the  convex non-smooth function $f:\R^2\To\R,\, f(x,y)=\sqrt{(x^2+y^2)^3}.$ Then, the proximal operator of $sf,\,s>0$ is given by $\prox\nolimits_{s f}:\R^2\To\R^2,$
$$\prox\nolimits_{s f}(x,y)=\left(\frac{2}{1+\sqrt{1+12 s\sqrt{x^2+y^2}}}x,\frac{2}{1+\sqrt{1+12 s\sqrt{x^2+y^2}}}y\right).$$
Consider further the function $g:\R^2\To\R,\, g(x,y)= (4x -3y)^2.$ Then, $g$ is smooth and convex and its gradient is Lipschitz continuous, having  Lipschitz constant $L=40\sqrt{2}.$

Further $f+g$ has the global minimum at $x^*=(0,0)$, hence  $\min(f+g)=(f+g)(x^*)=0.$  For simplicity in the following experiments concerning Algorithm \eqref{tdiscgen1} we take everywhere $s = 0.017$, (which will always satisfy $s<\frac{1}{L}$), and fix the starting points $x_0=x_1=(1, -1)$. In our following experiment we consider everywhere the Tikhonov regularization parameter $c=1.$

First, we consider both the cases $p<2q$ and $p\ge 2q$ with the following values:
$$(a,q,p)\in\{(0.9,0.5,0.9),(0.9,0.5,2.1),(0.66,0.75,1.4),(0.66,0.75,2.1),(0.45,1,2),(0.45,1,2.5)\}.$$
We run Algorithm \eqref{tdiscgen1} until the potential energy $(f+g)(x_n)-\min(f+g)$ is less than $10^{-20}$, the results are shown in Figure \ref{fig1:sfig11}.

According to our results, in case $q=1$ the optimal rate of order $\mathcal{O}(n^{-2})$ can be obtained. Therefore, we also consider the special case of our algorithm, i.e. \eqref{tdiscgen2}, and for comparison purposes we consider the FISTA algorithm \eqref{fista}, with $t_k=0.5k+1$.
We run the algorithms until the potential energy $(f+g)(x_n)-\min(f+g)$ is less than $10^{-25}$, the results are shown in Figure \ref{fig1:sfig12}.

\begin{figure}[hbt!]
\begin{subfigure}{0.5\textwidth}
  \centering
  \includegraphics[width=.99\linewidth]{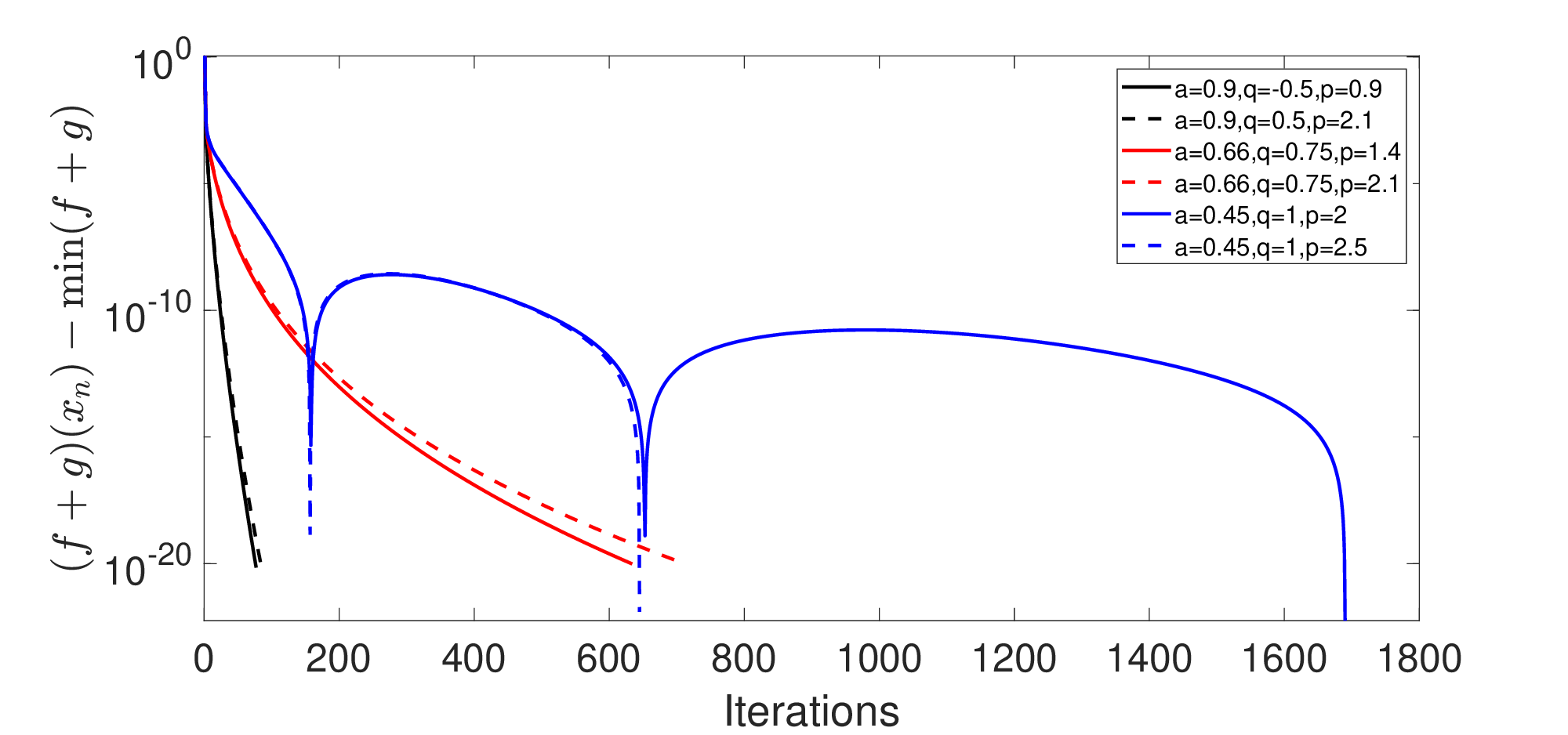}
  \caption{Different choices of the parameters in \eqref{tdiscgen1}}
  \label{fig1:sfig11}
\end{subfigure}
\begin{subfigure}{0.5\textwidth}
  \centering
  \includegraphics[width=.99\linewidth]{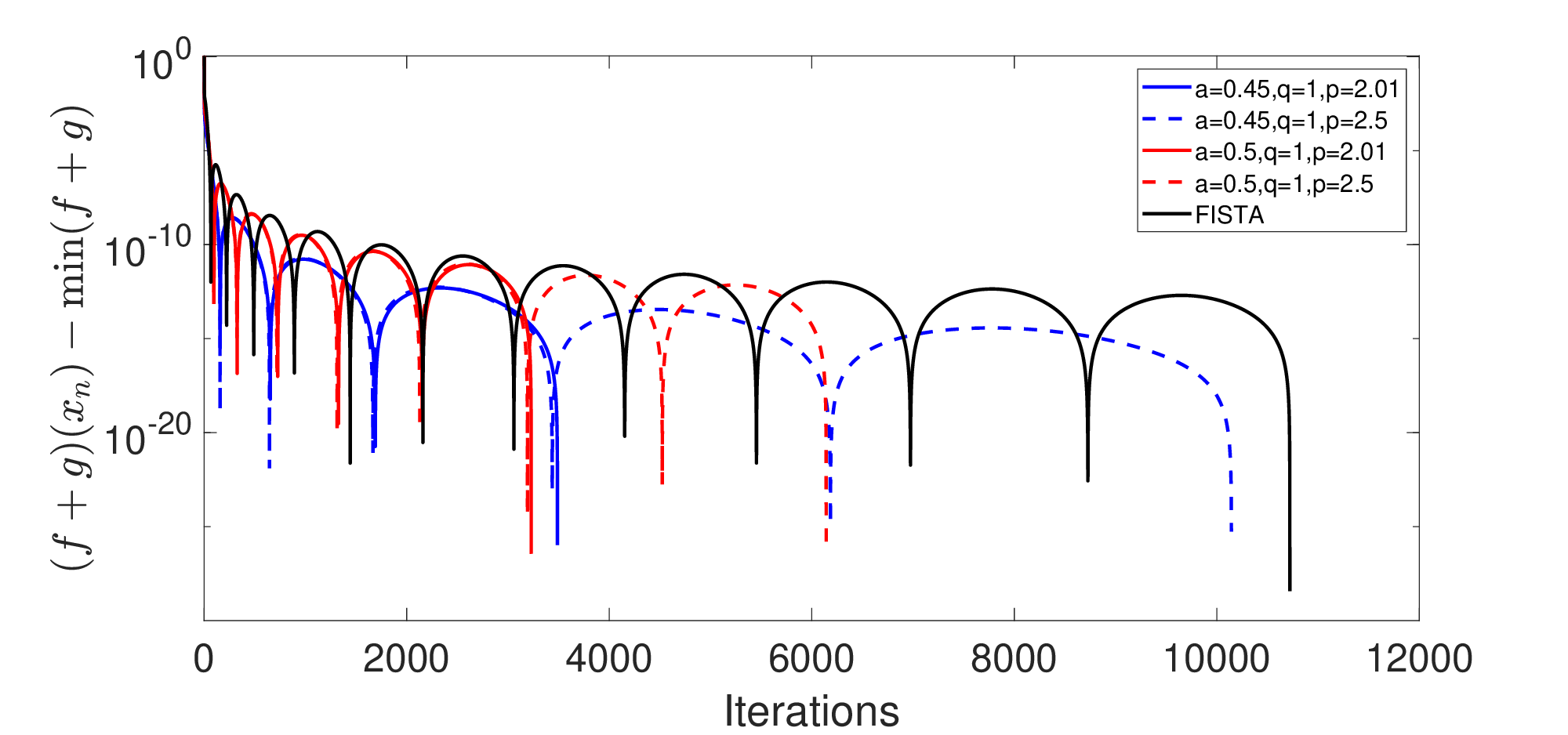}
  \caption{The case $q=1$ in \eqref{tdiscgen1} }
  \label{fig1:sfig12}
\end{subfigure}
 \caption{From a numerical point of view $q=1$ is not the best choice. Nevertheless, in case $q=1$ \eqref{tdiscgen1} might have a better behaviour than FISTA.}
\end{figure}
\vskip0.5cm

In our second experiment our aim is to show that in order to obtain convergence to the minimum norm minimizer of our objective function we need both Tikhonov regularization terms in \eqref{tdiscgen}. Hence, we consider the function $f:\R^2\to\R$, $f(x,y)=\frac12(x+2y)^2$ with
$$\prox\nolimits_{sf}(x,y)=\left(\frac{4s+1}{5s+1}x-\frac{2s}{5s+1}y,-\frac{2s}{5s+1}x+\frac{s+1}{5s+1}y\right).$$

Further, let $g:\R^2\to\R$, $g(x,y)=\frac14\left( x+2y\right)^2$ which is convex with $L$-Lipschitz continuous gradient.  Note that  the Lipschitz constant of $\n g$ is $L=\sqrt{10}.$ Obviously $\min(f+g)=0$ and the set of minimizers of $f+g$ is
$\argmin(f+g)=\left\{\left((x,-\frac{x}{2}\right):x\in\R\right\},$ hence the minimum norm minimizer is $x^*=(0,0).$

We take $s=0.031$ in algorithm \eqref{tdiscgen1} which always satisfies $s<\frac{1}{L}.$ Further we consider the starting points $x_0=x_1=(-1,1)$ and we fix $a=0.5,\,q=0.95,\, c=3$ and $p=1.7$ Note that $a<\frac{1}{2q}$ and $p<2q.$ We run \eqref{tdiscgen1} with $n=10^3$ iterations by considering the following instances.

a. We consider  both Tikhonov regularization terms in algorithm \eqref{tdiscgen1}, that is,
\begin{equation*}
\left\{\begin{array}{lll}
y_k= x_k+\frac{((ak+1)^q-1)((ak+1-a)^q-1)}{(ak+1-a)^{2q}}(x_k-x_{k-1})-\frac{-(ak+1)^{2q}+(ak+1)^q+(ak+1-a)^{2q}}{(ak+1-a)^{2q}(ak+1)^{q}} x_k
\\
x_{k+1}=\prox_{sf}\left(y_k-s\n g(y_k)-s\frac{c}{k^p} y_k\right).
\end{array}\right.
\end{equation*}

b. We renounce to the Tikhonov regularization term $\frac{-(ak+1)^{2q}+(ak+1)^q+(ak+1-a)^{2q}}{(ak+1-a)^{2q}(ak+1)^{q}} x_k$ in algorithm \eqref{tdiscgen1}, that is,
\begin{equation*}
\left\{\begin{array}{lll}
y_k= x_k+\frac{((ak+1)^q-1)((ak+1-a)^q-1)}{(ak+1-a)^{2q}}(x_k-x_{k-1})
\\
x_{k+1}=\prox_{sf}\left(y_k-s\n g(y_k)-s\frac{c}{k^p} y_k\right).
\end{array}\right.
\end{equation*}

c. We renounce to the Tikhonov regularization term $s\frac{c}{k^p} y_k$ in algorithm \eqref{tdiscgen1}, that is,
\begin{equation*}
\left\{\begin{array}{lll}
y_k= x_k+\frac{((ak+1)^q-1)((ak+1-a)^q-1)}{(ak+1-a)^{2q}}(x_k-x_{k-1})-\frac{-(ak+1)^{2q}+(ak+1)^q+(ak+1-a)^{2q}}{(ak+1-a)^{2q}(ak+1)^{q}} x_k
\\
x_{k+1}=\prox_{sf}\left(y_k-s\n g(y_k)\right).
\end{array}\right.
\end{equation*}

d. We renounce to both Tikhonov regularization terms in algorithm \eqref{tdiscgen1}, that is,
\begin{equation*}
\left\{\begin{array}{lll}
y_k= x_k+\frac{((ak+1)^q-1)((ak+1-a)^q-1)}{(ak+1-a)^{2q}}(x_k-x_{k-1})
\\
x_{k+1}=\prox_{sf}\left(y_k-s\n g(y_k)\right).
\end{array}\right.
\end{equation*}

The results are depicted at Figure 2 (a)-(d), where the first component of $x_n$ is depicted by blue and the second component of $x_n$  is depicted by red.

\begin{figure}[hbt!]
\begin{subfigure}{0.5\textwidth}
  \centering
  \includegraphics[width=.99\linewidth]{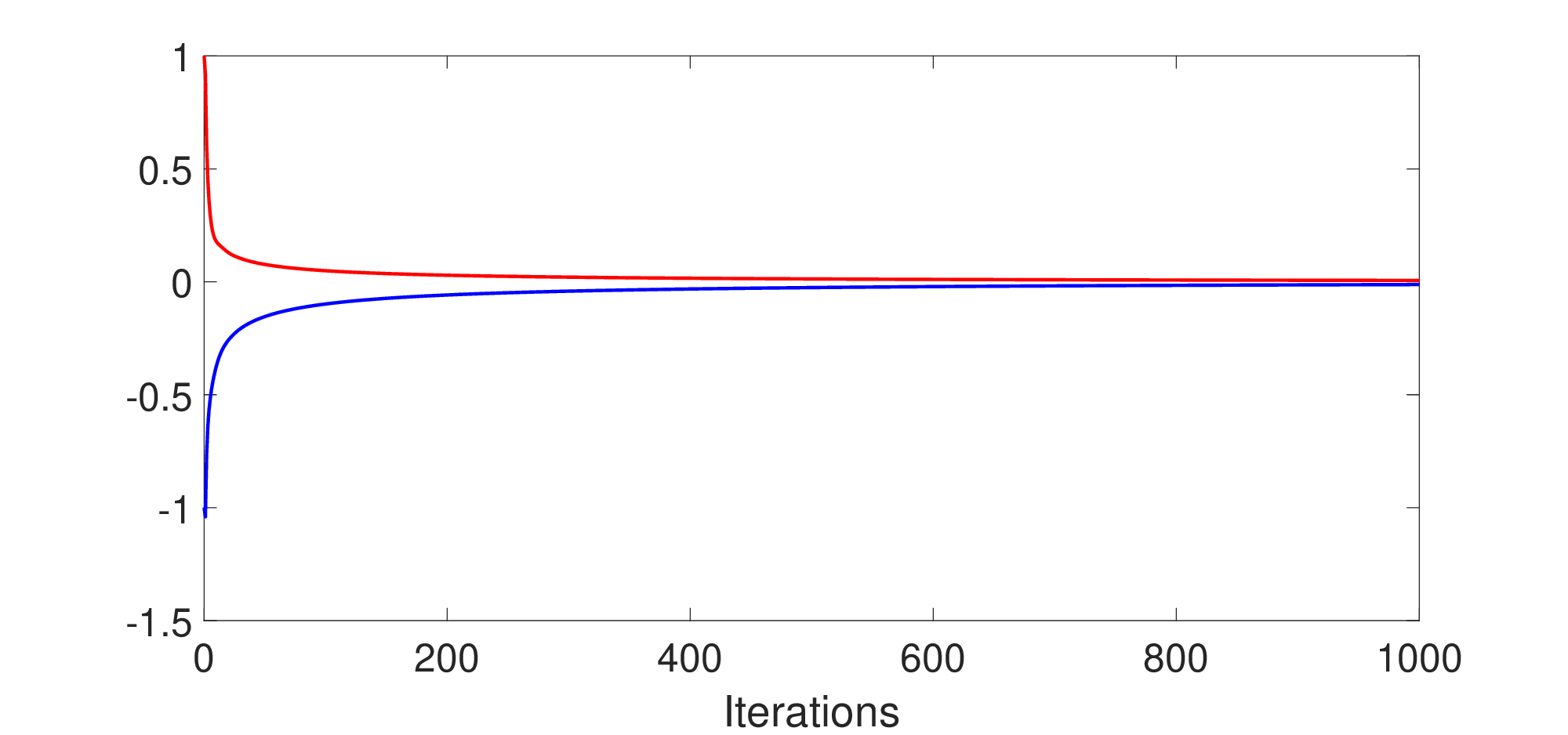}
  \caption{Convergence to the minimum norm solution}
  \label{fig2:sfig11}
\end{subfigure}
\begin{subfigure}{0.5\textwidth}
  \centering
  \includegraphics[width=.99\linewidth]{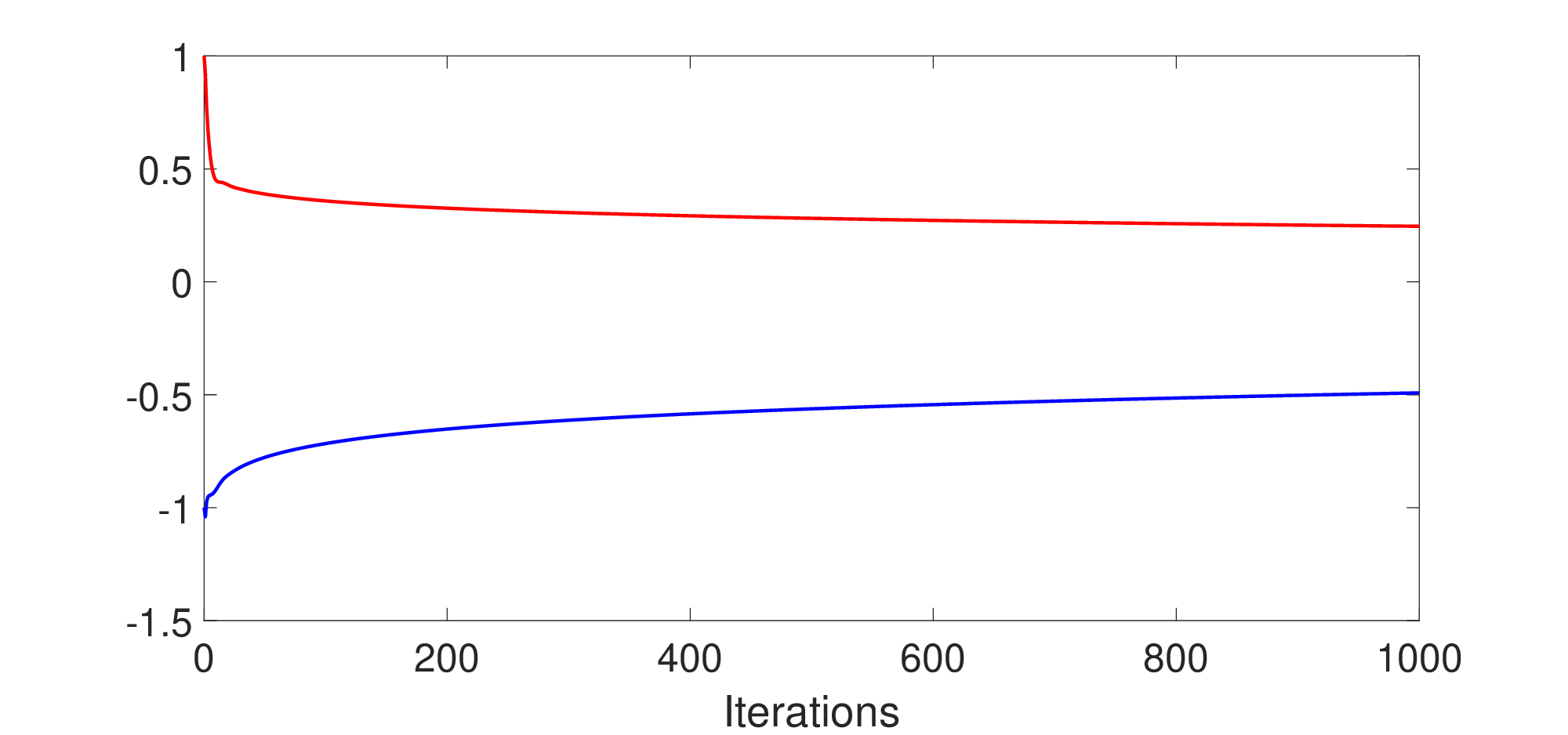}
  \caption{Dropping the first Tikhonov  term}
  \label{fig12sfig12}
\end{subfigure}

\begin{subfigure}{0.5\textwidth}
  \centering
  \includegraphics[width=.99\linewidth]{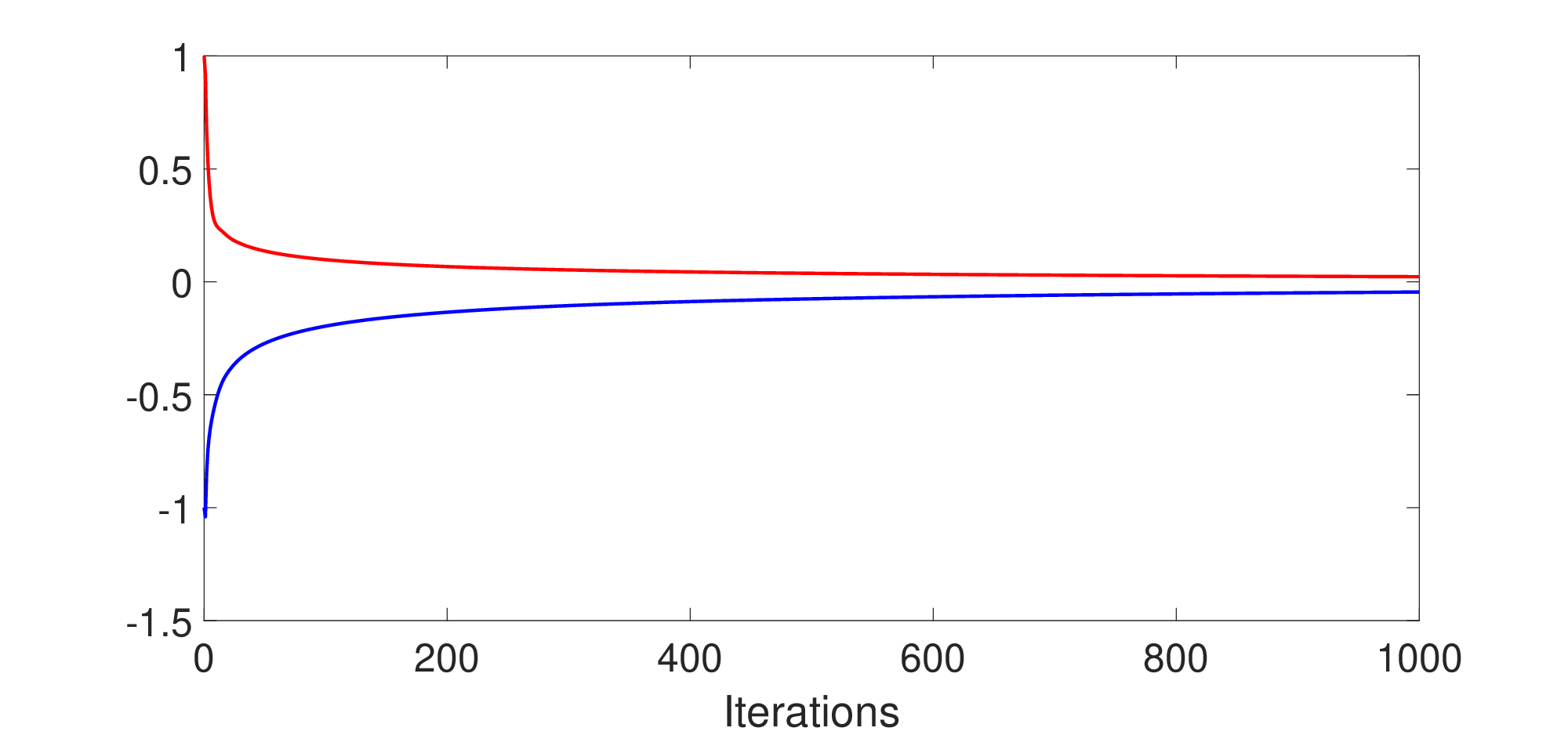}
  \caption{Dropping the second Tikhonov  term }
  \label{fig2:sfig13}
\end{subfigure}
\begin{subfigure}{0.5\textwidth}
  \centering
  \includegraphics[width=.99\linewidth]{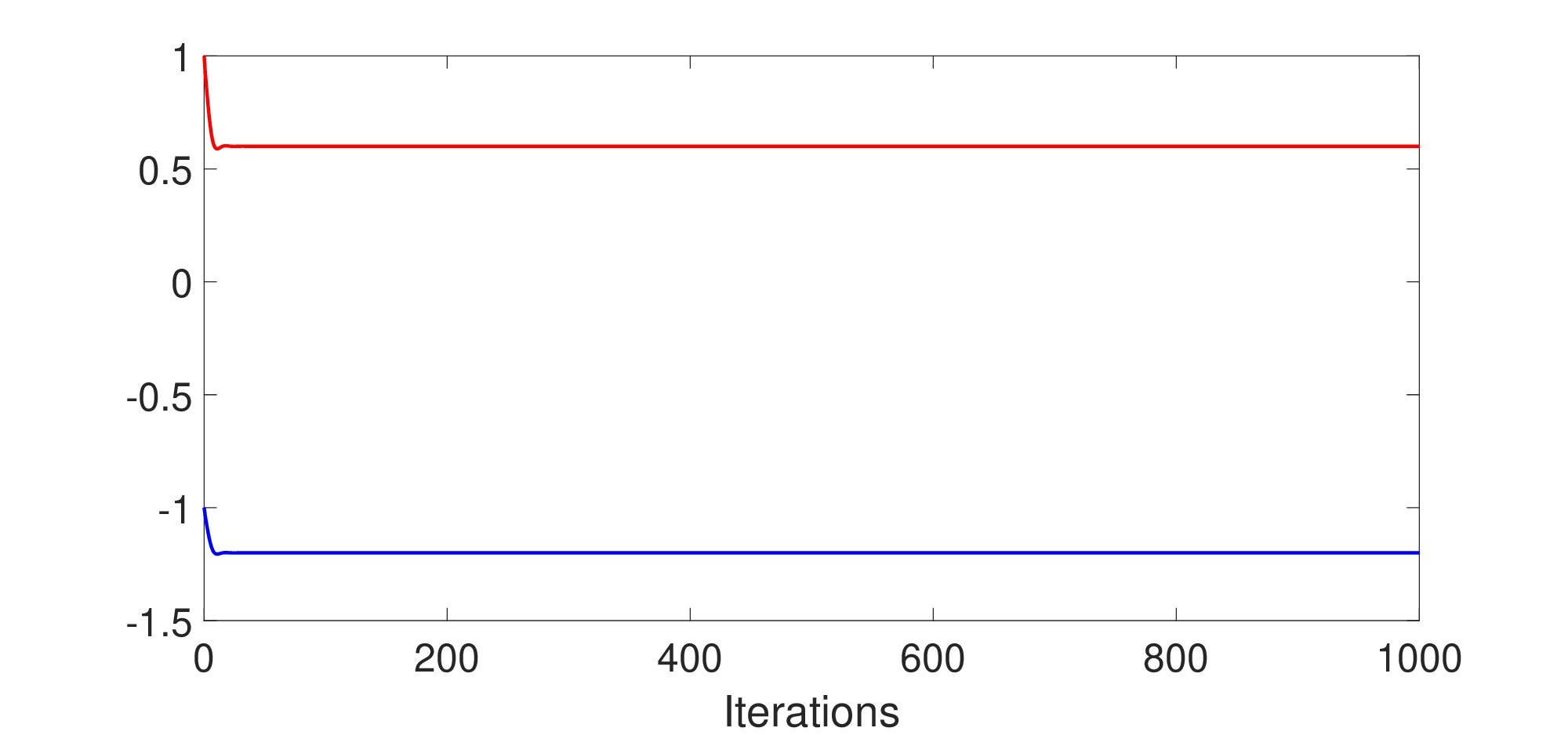}
  \caption{Dropping both Tikhonov  terms}
  \label{fig2:sfig14}
\end{subfigure}
 \caption{Renouncing to one of the Tikhonov regularization terms in \eqref{tdiscgen1} there is no convergence to the minimum norm solution anymore.}
\end{figure}

As we can see, in the absence of one of the Tikhonov regularization terms  we do not have convergence to the minimum norm minimizer of the objective function $f+g.$ Therefore,  according to the Figure 2,  the presence of both Tikhonov regularization terms in our algorithm \eqref{tdiscgen} is fully justified.

\section{Conclusions, perspectives}

To our best knowledge, Algorithm \eqref{tdiscgen} and in particular Algorithm \eqref{tdiscgen1} are the first inertial proximal-gradient type algorithms considered in the literature, that assure strong convergence to the minimum norm minimizer of the sum of a convex, possibly non-smooth and a smooth convex function, and at the same time provide fast convergence of the function values and  discrete velocity. As we have emphasized in the paper, these algorithms can be seen as FISTA type algorithms with two Tikhonov regularization terms. Further, we have shown that in order to obtain our strong convergence results we need both Tikhonov regularization terms in \eqref{tdiscgen}. Despite of the complex structure of the inertial parameter and the Tikhonov regularization parameters, our algorithms can easily be  implemented, therefore are suitable for use in practical problems arising in image processing and machine learning. Moreover, there are settings of the parameters when the optimal rate $\mathcal{O}(k^{-2})$ for the function values can be obtained, hence our algorithm can be a valuable substitute for FISTA.

A challenging related research is to consider the optimization problem having as objective the sum of a convex function and a smooth convex function, where the latter is composed with a linear operator and design an algorithm, (having in mind the ADMM algorithm with some Tikhonov regularization terms), in order to find the minimum norm solution.

\section{Declarations}

\begin{center}
    \textbf{Availability of data and materials}
\end{center}

In this manuscript only the datasets generated by authors were analysed.

\begin{center}
    \textbf{Competing interests}
\end{center}

The authors have no competing interests.

\section{Appendix}
\appendix

\section{The descent lemma and its consequences}
In order to obtain strong convergence   for the sequence $x_k$ generated by  Algorithm \eqref{tdiscgen} we need some preliminary results. The first one is the Descent Lemma \cite{Nest}.
\begin{lemma}\label{descent} Let $g:\mathcal{H}\To\R$ be a  smooth function, with $L_g-$Lipschitz continuous gradient. Then,
$$g(x)\le g(y)+\<\n g(y),x-y\>+\frac{L_g}{2}\|y-x\|^2,\,\forall x,y\in\mathcal{H}.$$
\end{lemma}

Further, we need the following property of smooth, convex functions,  see \cite{Nest}.

\begin{lemma}\label{gradLineq} Let $g:\mathcal{H}\To\R$ be a convex smooth function, with $L_g-$Lipschitz continuous gradient. Then,
$$\frac{1}{2L_g}\|\n g(y)-\n g(x)\|^2+\<\n g(y),x-y\>+g(y)\le g(x),\mbox{ for all }x,y\in\mathcal{H}.$$
\end{lemma}

The following modified descent lemma, which in particular contains Lemma 1 from \cite{ACFR}, has been proved in \cite{KL-amop}
\begin{lemma}\label{moddesc} Let $g:\mathcal{H}\To\R$ be a convex smooth function, with $L_g-$Lipschitz continuous gradient. Then,
\begin{equation}\label{f3}
g(y-s\n g(y))\le g(x)+\<\n g(y),y-x\>+\left(\frac{L_g}{2}s^2-s\right)\|\n g(y)\|^2-\frac{1}{2L_g}\|\n g(y)-\n g(x)\|^2,\,\forall x,y\in\mathcal{H}.
\end{equation}
Assume further that $s\in\left(0,\frac{1}{L_g}\right].$ Then,
\begin{equation}\label{f4}
g(y-s\n g(y))\le g(x)+\<\n g(y),y-x\>-\frac{s}{2}\|\n g(y)\|^2-\frac{s}{2}\|\n g(y)-\n g(x)\|^2,\,\forall x,y\in\mathcal{H}.
\end{equation}
\end{lemma}

Now, based on Lemma \ref{descent}, Lemma \ref{gradLineq} and Lemma \ref{moddesc} we give the following result which can be seen as an extension of Lemma 2.3 from \cite{BT}.

\begin{lemma}\label{modeBT} Let $f:\mathcal{H}\To\overline{\R}$ be a lower semi-continuous convex function and let  $g:\mathcal{H}\To\R$ be a convex smooth function, with $L_g-$Lipschitz continuous gradient. Then, for every $s>0$ and every $x,y\in\mathcal{H}$ one has
\begin{align}\label{f5}
(f+g)(\prox\nolimits_{sf}(y-s\n g(y))) &\le (f+g)(x)-\frac{1}{s}\left\<y-\prox\nolimits_{sf}(y-s\n g(y)),x-y\right\>\\
\nonumber&+\left(\frac{L_g}{2}-\frac{1}{s}\right)\|\prox\nolimits_{sf}(y-s\n g(y))-y\|^2-\frac{1}{2L_g}\|\n g(y)-\n g(x)\|^2.
\end{align}
Assume further that $s\in\left(0,\frac{1}{L_g}\right].$ Then, for all $x,y\in\mathcal{H}$ one has
\begin{align}\label{f6}
(f+g)(\prox\nolimits_{sf}(y-s\n g(y))) &\le (f+g)(x)-\frac{1}{s}\left\<y-\prox\nolimits_{sf}(y-s\n g(y)),x-y\right\>\\
\nonumber&-\frac{1}{2s}\|\prox\nolimits_{sf}(y-s\n g(y))-y\|^2-\frac{s}{2}\|\n g(y)-\n g(x)\|^2.
\end{align}
\end{lemma}

\begin{proof}
Indeed, by taking $x=\prox\nolimits_{sf}(y-s\n g(y))$ in Lemma \ref{descent}, we get
\begin{equation}\label{f1}
g(\prox\nolimits_{sf}(y-s\n g(y)))\le g(y)+\<\n g(y),\prox\nolimits_{sf}(y-s\n g(y))-y\>+\frac{L_g}{2}\|\prox\nolimits_{sf}(y-s\n g(y))-y\|^2,\,\forall y\in\mathcal{H}.
\end{equation}
From Lemma \ref{gradLineq}  we have
\begin{equation}\label{f2}
g(y)\le g(x)+\<\n g(y),y-x\>-\frac{1}{2L_g}\|\n g(y)-\n g(x)\|^2,\,\forall x,y\in\mathcal{H}.
\end{equation}

Combining \eqref{f1} and \eqref{f2} we get
\begin{align}\label{if}
g(\prox\nolimits_{sf}(y-s\n g(y)))&\le g(x)+\<\n g(y),\prox\nolimits_{sf}(y-s\n g(y))-x\>+\frac{L_g}{2}\|\prox\nolimits_{sf}(y-s\n g(y))-y\|^2\\
\nonumber&-\frac{1}{2L_g}\|\n g(y)-\n g(x)\|^2,\,\forall x,y\in\mathcal{H}.
\end{align}

Now, by using the fact that $\prox\nolimits_{sf}(z)=(I+s\p f)^{-1}(z),$ we get that $\frac{1}{s}(z-\prox\nolimits_{sf}(z))\in\p f(\prox\nolimits_{sf}(z))$ for all $z\in\mathcal{H}.$ Consequently, the sub-gradient inequality yields
$$f(x)\ge f(\prox\nolimits_{sf}(z))+\left\<\frac{1}{s}(z-\prox\nolimits_{sf}(z)),x-\prox\nolimits_{sf}(z)\right\>,\forall x,z\in\mathcal{H}.$$

Now by taking $z=y-s\n g(y)$ in the previous inequality, we get for all $x,y\in\mathcal{H}$ that
\begin{align}\label{if1}
f(\prox\nolimits_{sf}(y-s\n g(y))) \le f(x)-\frac{1}{s}\left\<y-s\n g(y)-\prox\nolimits_{sf}(y-s\n g(y)),x-\prox\nolimits_{sf}(y-s\n g(y))\right\>.
\end{align}
Now adding \eqref{if} and \eqref{if1} we get for all $x,y\in\mathcal{H}$ that
\begin{align}\label{if2}
(f+g)(\prox\nolimits_{sf}(y-s\n g(y))) &\le (f+g)(x)-\frac{1}{s}\left\<y-\prox\nolimits_{sf}(y-s\n g(y)),x-\prox\nolimits_{sf}(y-s\n g(y))\right\>\\
\nonumber&+\frac{L_g}{2}\|\prox\nolimits_{sf}(y-s\n g(y))-y\|^2-\frac{1}{2L_g}\|\n g(y)-\n g(x)\|^2\\
\nonumber&=(f+g)(x)-\frac{1}{s}\left\<y-\prox\nolimits_{sf}(y-s\n g(y)),x-y\right\>\\
\nonumber&+\left(\frac{L_g}{2}-\frac{1}{s}\right)\|\prox\nolimits_{sf}(y-s\n g(y))-y\|^2-\frac{1}{2L_g}\|\n g(y)-\n g(x)\|^2,
\end{align}
which is nothing else than \eqref{f5}.

Assume that $0<s\le\frac{1}{L_g}.$ Then,
$$\frac{L_g}{2}-\frac{1}{s}\le-\frac{1}{2s}
\mbox{ and }-\frac{1}{2L_g}\le-\frac{s}{2},$$
hence \eqref{f5} leads to \eqref{f6}.
\end{proof}

\section{Three pillars that sustain our results}\label{pil}

In this section we present three lemmas that are essential in order to prove our main result concerning the strong convergence of the sequences generated by \eqref{tdiscgen} to the minimum norm minimizer of our objective function $f+g.$

For $k\ge 1$ let us denote $g_k(\cdot)=g(\cdot)+\frac{\e_k}{2}\|\cdot\|^2$ and $F_k(\cdot)=f(\cdot)+g_k(\cdot)$. Then, both $g_k$ and $F_k$ are $\e_k-$strongly convex functions.
Moreover, since $g$ is smooth, $g_k$ is also smooth and $\n g_k(x)= \n g(x)+\e_k x$ for all $x\in\mathcal{H}$, further $\n g_k$ is also Lipschitz continuous having its Lipschitz constant $L+\e_k.$

In what follows we denote the unique minimizer of the strongly convex function $F_k$ by $\ol x_k.$

In order to obtain symmetry that allows us to apply telescopic sums the next result shows to be  crucial.

\begin{lemma}\label{forleftside}
Consider $x,y,z\in\mathcal{H}$ and let $(p_k)_{k\ge 1},\,(q_k)_{k\ge 1}$ be positive sequences. For every $k\ge 1$ consider  $z_k^*\in\p F_k(z)$. Then, for all $k\ge 1$ one has
\begin{align*}
(p_{k}+q_{k})F_k(x)-p_{k} F_k(y)-q_{k}F_k(z)\ge &(p_{k}+q_{k})(F_{k+1}(x)-F_{k+1}(\ol x_{k+1}))\\
&-(p_{k-1}+q_{k-1})(F_k(y)-F_k(\ol x_{k}))\\
\nonumber&+(p_{k-1}+q_{k-1}-p_{k})(F_k(y)-F_k(\ol x_{k}))\\
&+(p_{k}+q_{k})\frac{\e_{k+1}-\e_k}{2}\|\ol x_{k+1}\|^2+q_{k} \<z_k^*,\ol x_{k+1}-z\>\\
&+p_k\frac{\e_{k}}{2}\|\ox_{k+1}-\ox_k\|^2+q_k\frac{\e_k}{2}\|\ox_{k+1}-z\|^2.
\end{align*}
\end{lemma}
\begin{proof}
First of all note that
\begin{align*}
&(p_k+q_k)F_k(x)-p_k F_k(y)-q_kF_k(z)=(p_k+q_k)(F_k(x)-F_{k+1}(\ol x_{k+1})) +(p_k+q_k)F_{k+1}(\ol x_{k+1})\\
\nonumber&-p_k(F_k(y)-F_k(\ol x_{k}))-p_k F_k(\ol x_{k})-q_kF_k(z)
\end{align*}

Note that by using the fact that $(\e_k)$ is non-increasing we have
$$F_k(x)=F_{k+1}(x)+\frac{\e_k-\e_{k+1}}{2}\|x\|^2\ge F_{k+1}(x),$$ consequently it holds
\begin{align}\label{righttra}
&(p_k+q_k)F_k(x)-p_k F_k(y)-q_kF_k(z)=(p_k+q_k)(F_k(x)-F_{k+1}(\ol x_{k+1})) +(p_k+q_k)F_{k+1}(\ol x_{k+1})\\
\nonumber&-p_k(F_k(y)-F_k(\ol x_{k}))-p_k F_k(\ol x_{k})-q_kF_k(z)\ge (p_k+q_k)(F_{k+1}(x)-F_{k+1}(\ol x_{k+1}))\\
\nonumber& -(p_{k-1}+q_{k-1})(F_k(y)-F_k(\ol x_{k}))+(p_{k-1}+q_{k-1}-p_k)(F_k(y)-F_k(\ol x_{k}))\\
\nonumber&+p_k(F_{k+1}(\ol x_{k+1})-F_{k}(\ol x_{k}))+q_k(F_{k+1}(\ol x_{k+1})-F_k(z)).
\end{align}
Let us further estimate the terms $p_k(F_{k+1}(\ol x_{k+1})-F_{k}(\ol x_{k}))$ and $q_k(F_{k+1}(\ol x_{k+1})-F_k(z))$ in \eqref{righttra}.

On one hand, according to \eqref{fontos3} one has $F_{k}(\ox_{k+1})-F_{k}(\ox_{k})\ge\frac{\e_{k}}{2}\|\ox_{k+1}-\ox_k\|^2$ hence
\begin{align}\label{righttrala}
p_k(F_{k+1}(\ol x_{k+1})-F_{k}(\ol x_{k}))&=p_k\left(F_{k}(\ol x_{k+1})-F_{k}(\ol x_{k})+\frac{\e_{k+1}-\e_k}{2}\|\ol x_{k+1}\|^2\right)\\
\nonumber&\ge p_k\frac{\e_{k}}{2}\|\ox_{k+1}-\ox_k\|^2+p_k\frac{\e_{k+1}-\e_k}{2}\|\ol x_{k+1}\|^2\ge p_k\frac{\e_{k+1}-\e_k}{2}\|\ol x_{k+1}\|^2.
\end{align}
On the other hand, by using the sub-gradient inequality  we get for every $z_k^*\in\p F_k(z)$ that
\begin{align}\label{righttralala}
&q_k(F_{k+1}(\ol x_{k+1})-F_k(z))=q_k\left(F_{k}(\ol x_{k+1})-F_k(z)+\frac{\e_{k+1}-\e_k}{2}\|\ol x_{k+1}\|^2\right)\ge\\
\nonumber&q_k \<z_k^*,\ol x_{k+1}-z\>+q_k\frac{\e_k}{2}\|\ox_{k+1}-z\|^2+q_k\frac{\e_{k+1}-\e_k}{2}\|\ol x_{k+1}\|^2\ge\\
 \nonumber&q_k\<z_k^*,\ol x_{k+1}-z\>+q_k\frac{\e_{k+1}-\e_k}{2}\|\ol x_{k+1}\|^2.
\end{align}

Consequently, \eqref{righttra}, \eqref{righttrala} and \eqref{righttralala} lead to
\begin{align}\label{righttralalala}
&(p_k+q_k)F_k(x)-p_k F_k(y)-q_kF_k(z)\ge\\
\nonumber& (p_k+q_k)(F_{k+1}(x)-F_{k+1}(\ol x_{k+1}))-(p_{k-1}+q_{k-1})(F_k(y)-F_k(\ol x_{k}))\\
\nonumber&+(p_{k-1}+q_{k-1}-p_k)(F_k(y)-F_k(\ol x_{k}))+(p_k+q_k)\frac{\e_{k+1}-\e_k}{2}\|\ol x_{k+1}\|^2+q_k \<z_k^*,\ol x_{k+1}-z\>\\
\nonumber&+p_k\frac{\e_{k}}{2}\|\ox_{k+1}-\ox_k\|^2+q_k\frac{\e_k}{2}\|\ox_{k+1}-z\|^2.
\end{align}
\end{proof}

\begin{remark}\label{forleftsidenoeps}
Note that in case we deal only with the convex function $F=f+g$, i.e., we take the parameter $\e_k\equiv 0$, then some similar result to those obtained in  Lemma \ref{forleftside} holds.
Indeed, in that case one can simply write for every $\ol x\in\argmin F$ that
\begin{align*}
(p_{k}+q_{k})F(x)-p_{k} F(y)-q_{k}F(z)= &(p_{k}+q_{k})(F(x)- F(\ol x))-(p_{k-1}+q_{k-1})(F(y)-F(\ol x))\\
\nonumber&+(p_{k-1}+q_{k-1}-p_{k})(F(y)- F(\ol x))+q_k( F(\ol x)-F(z))
\end{align*}
and by the sub-gradient inequality one has
$q_k( F(\ol x)-F(z))\ge q_{k} \<z^*,\ol x-z\>\mbox{ for all }z^*\in\p F(z).$
Hence,
\begin{align*}
(p_{k}+q_{k})F(x)-p_{k} F(y)-q_{k}F(z)\ge &(p_{k}+q_{k})(F(x)- F(\ol x))-(p_{k-1}+q_{k-1})(F(y)-F(\ol x))\\
\nonumber&+(p_{k-1}+q_{k-1}-p_{k})(F(y)- F(\ol x))+q_{k} \<z^*,\ol x-z\>.
\end{align*}
Nevertheless, in case there are no Tikhonov regularization terms, as in case of FISTA, it is  enough to consider the following identity:
\begin{align}\label{leftsidenoeps}
(p_{k}+q_{k})F(x)-p_{k} F(y)-q_{k}\min F= &(p_{k}+q_{k})(F(x)- \min F)-(p_{k-1}+q_{k-1})(F(y)-\min F)\\
\nonumber&+(p_{k-1}+q_{k-1}-p_{k})(F(y)- \min F).
\end{align}
\end{remark}

The next lema deals with some affine combinations of the sequences generates by \eqref{tdiscgen} and will be used in the proof of our main result.

\begin{lemma}\label{foretan1} Consider the sequence $(t_k)_{k\ge 0}$ that satisfies \eqref{condT} and let $(x_k)_{k\ge 1}$ and $(y_k)_{k\ge 1}$ be the sequences generated by Algorithm \eqref{tdiscgen}.
For $k\ge 1$, consider the sequence  $\eta_k=\frac{t_{k-1}^2}{t_k-1}y_k-\frac{t_{k-1}^2}{t_k}x_k.$

Then, $(\eta_k)_{k\ge 1}$ satisfies the following recursion:
\begin{equation}\label{reletan}
\eta_{k+1}=\frac{t_k^2-t_k}{t_{k-1}^2}\eta_k+t_k(x_{k+1}-y_k).
\end{equation}
Moreover, $\eta_{k+1}$ is an affine combination of $x_{k+1}$ and $x_k$, that is
\begin{equation}\label{convcombetan}
\eta_{k+1}=t_kx_{k+1}+(1-t_k)x_k.
\end{equation}
\end{lemma}
\begin{proof}
We have to show that
$\eta_{k+1}=\frac{t_{k}^2}{t_{k+1}-1}y_{k+1}-\frac{t_{k}^2}{t_{k+1}}x_{k+1}=\frac{t_k^2-t_k}{t_{k-1}^2}\eta_k+t_k(x_{k+1}-y_k).$

On one hand, we have
 $$\frac{t_k^2-t_k}{t_{k-1}^2}\eta_k+t_k(x_{k+1}-y_k)=\frac{t_k^2-t_k}{t_{k-1}^2}\left( \frac{t_{k-1}^2}{t_k-1}y_k-\frac{t_{k-1}^2}{t_k}x_k \right)+t_k(x_{k+1}-y_k)=t_kx_{k+1}+(1-t_k)x_k.$$

On the other hand, according to \eqref{tdiscgen},  we have
\begin{align*}
y_{k+1}&=x_{k+1}+\frac{(t_{k+1}-1)(t_{k}-1)}{t_{k}^2}(x_{k+1}-x_{k})-\frac{-t_{k+1}^2+t_{k+1}+t_{k}^2}{t_{k}^2t_{k+1}} x_{k+1}\\
&=\frac{(t_{k+1}+t_k)(t_{k+1}-1)}{t_kt_{k+1}}x_{k+1}-\frac{(t_{k+1}-1)(t_k-1)}{t_k^2}x_k,
\end{align*}
hence,
\begin{align*}
\eta_{k+1}&=\frac{t_{k}^2}{t_{k+1}-1}y_{k+1}-\frac{t_{k}^2}{t_{k+1}}x_{k+1}\\
&=\frac{t_{k}^2}{t_{k+1}-1}\left(\frac{(t_{k+1}+t_k)(t_{k+1}-1)}{t_kt_{k+1}}x_{k+1}-\frac{(t_{k+1}-1)(t_k-1)}{t_k^2}x_k\right)-\frac{t_{k}^2}{t_{k+1}}x_{k+1}\\
&=t_kx_{k+1}+(1-t_k)x_k.
\end{align*}
\end{proof}

The next lemma will be very useful in the proof of our main result.

\begin{lemma}\label{foretandif} Consider the sequence $(t_k)_{k\ge 0}$ that satisfies \eqref{condT} and let $(x_k)_{k\ge 1}$ and $(y_k)_{k\ge 1}$ be the sequences generated by Algorithm \eqref{tdiscgen}.
For $k\ge 1$, consider the sequence  $\eta_k=\frac{t_{k-1}^2}{t_k-1}y_k-\frac{t_{k-1}^2}{t_k}x_k$ and let $x^*\in\mathcal{H}.$
Then, the following identity holds.
\begin{align}\label{leftreldif}
&
\left\<y_k-x_{k+1},2t_k^2y_k-2(t_k^2-t_k) x_k-2t_k x^*-t_k^2(y_k-x_{k+1})\right\>=\frac{t_k^2-t_k}{t_{k-1}^2}\|\eta_k-x^*\|^2-\|\eta_{k+1}-x^*\|^2\\
\nonumber&-\frac{t_k^2-t_k}{t_{k-1}^2}\cdot\frac{-t_k^2+t_k+t_{k-1}^2}{t_{k-1}^2}\|\eta_k\|^2+\frac{-t_k^2+t_k+t_{k-1}^2}{t_{k-1}^2}\|x^*\|^2,
\end{align}
for all $k\ge 1.$
\end{lemma}
\begin{proof}
For simplicity, let us denote $\a_k=\frac{-t_k^2+t_k+t_{k-1}^2}{t_{k-1}^2},\,k\ge 1.$ Since $(t_k)_{k\ge 0}$ that satisfies \eqref{condT}, we have  $t_k\in\left(1,\frac{1+\sqrt{1+4t_{k-1}^2}}{2}\right)$ for all $k\ge 1$, that is $-t_k^2+t_k+t_{k-1}^2>0$ for all $k\ge 1$, we deduce that
$$1>\a_k> 0.$$
Hence, the right hand side of \eqref{leftreldif} can be written as
$$(1-\a_k)\|\eta_k-x^*\|^2-\|\eta_{k+1}-x^*\|^2-(1-\a_k)\a_k\|\eta_k\|^2+\a_k\|x^*\|^2.$$

According to Lemma \ref{foretan1}, for all $k\ge 1$ one has
$\eta_{k+1}=(1-\a_k)\eta_k+t_k(x_{k+1}-y_k),$
hence
\begin{align}\label{fontos}
&(1-\a_k)\|\eta_k-x^*\|^2-\|\eta_{k+1}-x^*\|^2+\a_k\|x^*\|^2=(1-\a_k)\|\eta_k\|^2+2\<\eta_{k+1}-(1-\a_k)\eta_k,x^*\>-\|\eta_{k+1}\|^2\\
\nonumber&=(1-\a_k)\|\eta_k\|^2+2\<\eta_{k+1}-(1-\a_k)\eta_k,x^*\>-\|\eta_{k+1}\|^2\\
\nonumber&=(1-\a_k)\|\eta_k\|^2+\left\<y_k-x_{k+1},-2t_kx^*\right\>-\left\|(1-\a_k)\eta_k+t_k(x_{k+1}-y_k)\right\|^2\\
\nonumber&=(1-\a_k)\a_k\|\eta_k\|^2+\left\<y_k-x_{k+1},-2t_kx^*\right\>+\left\<y_k-x_{k+1},2t_k(1-\a_k)\eta_k\right\>-t_k^2\|y_k-x_{k+1}\|^2\\
\nonumber&=(1-\a_k)\a_k\|\eta_k\|^2+\left\<y_k-x_{k+1},-2t_kx^*+2t_k(1-\a_k)\eta_k-t_k^2(y_k-x_{k+1})\right\>
\end{align}
Note that $$2t_k(1-\a_k)\eta_k=2t_k\cdot \frac{t_k^2-t_k}{t_{k-1}^2}\left( \frac{t_{k-1}^2}{t_k-1}y_k-\frac{t_{k-1}^2}{t_k}x_k\right)=2t_k^2y_k-2(t_k^2-t_k)x_k,$$
hence \eqref{fontos} becomes
\begin{align}\label{fontos1}
(1-\a_k)\|\eta_k-x^*\|^2-\|\eta_{k+1}&-x^*\|^2+\a_k\|x^*\|^2=(1-\a_k)\a_k\|\eta_k\|^2\\
\nonumber&+\left\<y_k-x_{k+1},2t_k^2y_k-2(t_k^2-t_k) x_k-2t_k x^*-t_k^2(y_k-x_{k+1})\right\>
\end{align}
which is exactly our claim.
\end{proof}

\begin{remark}\label{reletannoeps}
Consider now that we take $\e_k\equiv 0$ in \eqref{tdiscgen}. Then $y_k$ remains unchanged and therefore all the conclusions of Lemma \ref{foretan1} and Lemma \ref{foretandif} remain valid.
\end{remark}

\section{On the Tikhonov regularization techniques}

We continue the present section by emphasizing the main idea behind the Tikhonov regularization, which will assure strong convergence results for the sequence generated our algorithm \eqref{tdiscgen1} to a minimizer of the objective function of minimal norm. By  $\ox_{k}$ we denote the unique solution of the strongly convex minimization problem
\begin{align*}
 \min_{x \in \mathcal{H}} \left((f+ g)(x) + \frac{\e_k}{2} \| x \|^2 \right).
\end{align*}
We know, (see for instance \cite{att-com1996}), that $\lim\limits_{k \to +\infty} \ox_{k}=x^\ast$, where $x^\ast = \argmin\limits_{x \in \argmin (f+g)} \| x \|$ is the minimal norm element from the set $\argmin (f+g).$ Obviously, $\{x^*\}=\pr_{\argmin(f+ g)} 0$ and we have the inequality $\| \ox_{k} \| \leq \| x^\ast \|$ (see \cite{BCL}).

Since $\ox_{k}$ is the unique minimum of the strongly convex function $F_k(x)=f(x)+g(x)+\frac{\e_k}{2}\|x\|^2,$ obviously one has
\begin{equation}\label{fontos0}
\p F_k(\ox_{k})=\p f(\ox_k)+\n g(\ox_{k})+\e_k\ox_{k}\ni0.
\end{equation}
Further, from Lemma A.1 c) from \cite{L-CNSN} we have
\begin{equation}\label{lfos}
\left\|\ox_{k+1}-\ox_k\right\|\le\min\left(\frac{\e_k-\e_{k+1}}{\e_{k+1}}\|\ox_{k}\|,\frac{\e_k-\e_{k+1}}{\e_k}\|\ox_{k+1}\|\right).
\end{equation}

Note that since $F_k$ is strongly convex, from the gradient inequality we have
\begin{equation}\label{fontos2}
F_k(y)-F_k(x)\ge\<u_k,y-x\>+\frac{\e_k}{2}\|x-y\|^2,\mbox{ for all }u_k\in\p F_k(x)\mbox{ and }x,y\in\mathcal{H}.
\end{equation}
In particular
\begin{equation}\label{fontos3}
F_k(x)-F_k(\ox_k)\ge\frac{\e_k}{2}\|x-\ox_k\|^2,\mbox{ for all }x\in\mathcal{H}.
\end{equation}

Not that the latter relation may assure the strong convergence to the minimum norm solution $x^*$ of the sequences generated by \eqref{tdiscgen}. Indeed, let $(x_k)$ be the sequence generated by \eqref{tdiscgen} and assume that
$$\lim_{k\to+\infty}\frac{F_k(x_k)-F_k(\ox_k)}{\e_k}=0.$$
Then, $\lim_{k\to+\infty}\|x_k-\ox_k\|=0$ which combined with the fact that $\lim\limits_{k \to +\infty} \ox_{k}=x^\ast$ leads to
$$\lim\limits_{k \to +\infty} x_{k}=x^\ast=0.$$

Moreover, observe that for all $x,y\in\mathcal{H}$, one has

\begin{equation}\label{fontos5}
(f+g)(x)-(f+g)(y)=(F_k(x)-F_k(\ox_k))+(F_k(\ox_k)-F_k(y))+\frac{\e_k}{2}(\|y\|^2-\|x\|^2)\le F_k(x)-F_k(\ox_k)+\frac{\e_k}{2}\|y\|^2.
\end{equation}

Consequently, if one already has a rate for $F_k(x_k)-F_k(\ox_k)$, where  $(x_k)$ is the sequence generated by \eqref{tdiscgen}, then \eqref{fontos5} provides a rate for the potential energy $(f+g)(x_k)-\min (f+g).$
Indeed, one has
$$(f+g)(x_k)-\min (f+g)\le F_k(x_k)-F_k(\ox_k)+\frac{\e_k}{2}\|x^*\|^2,$$
where  $x^*$  is the minimum norm minimizer of $f+g$.

\end{document}